\documentclass[11pt,reqno]{amsart}

\usepackage{amsmath}
\usepackage{amssymb}
\usepackage{amsthm}

\newtheorem{theorem}{Theorem}[section]
\newtheorem{prop}[theorem]{Proposition}
\newtheorem{lem}[theorem]{Lemma}
\newtheorem*{cor}{Corollary}
\theoremstyle{definition}
\newtheorem{defn}[theorem]{Definition}

\theoremstyle{remark}
\newtheorem*{rem}{Remark}
\numberwithin{equation}{section}

\begin{document}

\title[Locally compact quantum groupoids]
{A class of {$C^*$-algebraic} locally compact quantum groupoids \\
Part I. Motivation and definition}

\author{Byung-Jay Kahng}
\date{}
\address{Department of Mathematics and Statistics\\ Canisius College\\
Buffalo, NY 14208, USA}
\email{kahngb@canisius.edu}

\author{Alfons Van Daele}
\date{}
\address{Department of Mathematics\\ University of Leuven\\ Celestijnenlaan 200B\\ 
B-3001 Heverlee, BELGIUM}
\email{Alfons.VanDaele@wis.kuleuven.be}


\begin{abstract}
In this series of papers, we develop the theory of a class of locally compact quantum 
groupoids, which is motivated by the purely algebraic notion of weak multiplier 
Hopf algebras.  In this Part~I, we provide motivation and formulate the definition 
in the $C^*$-algebra framework.  Existence of a certain canonical idempotent element 
is required and it plays a fundamental role, including the establishment of the coassociativity 
of the comultiplication.  This class contains locally compact quantum groups as a subclass.
\end{abstract}
\maketitle

{\sc Introduction.}

Let $G$ be a {\em groupoid\/} over the set of units $G^{(0)}$, together with the maps 
$s:G\to G^{(0)}$ and $t:G\to G^{(0)}$ (the ``source map'' and the ``target map'').  This means 
that there is a set of ``composable pairs'' $G^{(2)}=\bigl\{(p,q)\in G\times G: s(p)=t(q)\bigr\}$, 
on which the product $pq$ in $G$ is defined.  This product is assumed to be associative, 
in an appropriate sense.  The set of units, $G^{(0)}$, may be naturally regarded as a subset 
of $G$.  There exists also the inverse map $p\mapsto p^{-1}$ (so that $(p^{-1})^{-1}=p$), 
for which we have $s(p^{-1})=t(p)$, $t(p^{-1})=s(p)$, and satisfying some natural conditions. 
When $G^{(0)}$ is a single element set, then $G$ becomes an ordinary group.  For a more detailed 
discussion on the definition and the basic theory of groupoids, refer to \cite{Brbook},  
\cite{Higbook}.  The groupoid notion can be further extended to incorporate locally compact 
topology, which is the notion of a {\em locally compact groupoid\/}.  See \cite{Renbook}, 
\cite{Patbook}.

Our main goal is to develop an appropriate notion of a locally compact quantum groupoid, which 
would generalize the category of locally compact quantum groups.  At the purely algebraic level, 
attempts in this direction have been around for some time, with works by Maltsiniotis \cite{MalGQ} 
(when the unit space is commutative), Lu \cite{LuAlgebroid} (``Hopf algebroids''), Xu \cite
{XuQGDeformation} (``quantum universal enveloping algebroids''), for instance.  In a more 
operator algebraic setting, there is a work by Yamanouchi \cite{Yagroupoid} (``generalized 
Kac algebras'').   Among these, our approach below is closely related to the notion of {\em weak 
Hopf algebras\/} (by B{\" o}hm, Nill, Szlach{\' a}nyi \cite{BNSwha1}, \cite{BSwha2}), also referred to 
as {\em finite quantum groupoids\/}.  See also \cite{NVfqg}, \cite{Valfqg}.

Suppose $G$ is a groupoid.  For the time being, let us disregard any topology on $G$ and 
consider $A=K(G)$, the set of all complex-valued functions on $G$ having finite support.  Under 
the pointwise multiplication, $A$ becomes a commutative algebra.  In particular, if $G$ 
is a finite groupoid, the algebra $A$ becomes unital, whose unit is denoted by 1.  There 
exists a map $\Delta:A\to A\otimes A$ (the ``comultiplication''), defined by
$$
(\Delta f)(p,q):=\left\{\begin{matrix}f(pq) & {\text { if $s(p)=t(q)$}} \\ 0 & {\text { otherwise}}
\end{matrix}\right.
$$
Here, $\otimes$ denotes the algebraic tensor product.  The ``antipode'' map is given by $S(f)(p)
=f(p^{-1})$.  In this case, it can be shown that $(A,\Delta)$ is given a structure of a weak 
Hopf algebra.  

A weak Hopf algebra is in general noncommutative, but includes the above example as 
a fundamental case.  Recently, this notion has been further generalized to include the case 
of non-finite groupoids (so the algebra $A$ is non-unital), which is the notion of a {\em weak 
multiplier Hopf algebra\/} developed by one of us (Van~Daele) and Wang.  Here, we have 
$\Delta:A\to M(A\otimes A)$, into the multiplier algebra of $A\otimes A$.  For fuller 
description, refer to \cite{VDWangwha0}, \cite{VDWangwha1}.

We are interested in developing its topological version, in the $C^*$-algebra framework, which 
would give us a {\em locally compact quantum groupoid\/}.  But first, let us review a little about 
locally compact groupoids.  In the standard theory on locally compact groupoids (see  \cite{Renbook}, 
\cite{Patbook}), in addition to requiring that the topology on $G$ is compatible with the groupoid 
structure on $G$ (including the unit space, the source and target maps, the multiplication 
and the inverse map), one also requires the existence of a suitable ``left Haar system''.

Briefly speaking, a left Haar system for a locally compact groupoid $G$ is a family 
$\{\lambda^u\}_{u\in G^{(0)}}$, where each $\lambda^u$ is a positive regular Borel measure 
on $G^u:=t^{-1}\{u\}$, such that 
\begin{itemize}
  \item For each $g\in C_c(G)$, the function $u\mapsto\int_{G^u}g\,d\lambda^u$ belongs to $C_c(G^{(0)})$;
  \item For any $x\in G$ and $f\in C_c(G)$, we have:
\begin{equation}\label{(groupoid_leftinvariance)}
\int_{G^{s(x)}}f(xz)\,d\lambda^{s(x)}(z)=\int_{G^{t(x)}}f(y)d\lambda^{t(x)}(y).
\end{equation}
\end{itemize}
Equation~\eqref{(groupoid_leftinvariance)} is referred to as the {\em left invariance condition\/}. 
When $G$ is a locally compact group, so $G^{(0)}=\{e\}$, this degenerates into the familiar 
left invariance condition for the left Haar measure on $G$.  In the left side of the equation, 
$z\in G^{s(x)}=t^{-1}\bigl\{s(x)\bigr\}$, because $xz$ is valid only when $t(z)=s(x)$.  In the right 
hand side, we can observe that $x\mapsto\int_{G^{t(x)}}f(y)d\lambda^{t(x)}(y)$ is a function 
of $t(x)$, so contained in the ``target algebra''.

Meanwhile, using the inverse map $x\mapsto x^{-1}$, which is a homeomorphism on $G$, 
one can associate to each $\lambda^u$ the measure $\lambda_u$ on $G_u:=s^{-1}\{u\}$.  
That is, $\lambda_u(V)=\lambda^u(V^{-1})$, for any Borel subset $V$ of $G_u$.  In this way, 
one obtains a ``right Haar system'' $\{\lambda_u\}$.  It would satisfy the {\em right invariance 
condition\/}, analogous to Equation~\eqref{(groupoid_leftinvariance)}.

Since the unit space is no longer a single element set, one also needs to consider a measure, 
$\nu$, on $G^{(0)}$.  However, for the theory to be properly developed, one needs to require 
what is called a ``quasi-invariance condition'' on $\nu$.  To explain, note that the measure $\nu$ 
on $G^{(0)}$ and the left Haar system determine a measure on $G$, which in turn will determine 
a functional (weight) $\varphi$, as follows:
\begin{equation}\label{(groupoid_leftHaar)}
\varphi(f)=\int_{G^{(0)}}\,d\nu(u)\int_{G^u}f(x)\,d\lambda^u(x).
\end{equation}
Similarly, the measure $\nu$ and the right Haar system will determine another measure on $G$, 
which in turn will produce a weight $\psi$.  The measure $\nu$ is called {\em quasi-invariant\/}, if 
the two resulting measures on $G$, namely $\varphi$ and $\psi$,  are mutually absolutely continuous.  
All these are standard results.  See \cite{Renbook}, \cite{Patbook} for more details.  

The above review on locally compact groupoids means that the data we expect to have for a general 
locally compact quantum groupoid are $(A,\Delta,B,\nu,S,\varphi)$, where $A$ is a $C^*$-algebra, 
taking the role of $C_0(G)$; $\Delta$ is the comultiplication on $A$, corresponding to the 
multiplication on $G$; $B$ is a sub-$C^*$-algebra of $M(A)$, taking the role of the unit space 
$G^{(0)}$; $\nu$ is a suitable weight on $B$; $S$ is the antipode map, corresponding to 
the inverse map on $G$; $\varphi$ is the left Haar weight on $A$, incorporating the information 
about the left Haar system and the left invariance. 

Or, equivalently, we may consider $(A,\Delta,B,\nu,\varphi,\psi)$, where $\psi$ is the right 
Haar weight.  The weights $\varphi$ and $\psi$ would be related via $S$, so we expect to recover 
the antipode $S$ from $\varphi$ and $\psi$.  This approach is analogous to what was done 
in \cite{KuVa} and \cite{KuVavN} in the theory of locally compact quantum groups.  In fact, 
in a separate paper by the authors \cite{BJKVD_LSthm}, we were able to carry out this exact idea 
in the purely algebraic setting of weak multiplier Hopf algebras.

In the general theory of weak multiplier Hopf algebras, a fundamental role is played by a certain 
canonical idempotent element $E$, which is just $\Delta(1)$ if $A$ has identity.  To get a sense 
of this, consider the special case $A=K(G)$ corresponding to a groupoid $G$, and consider
$$
E(p,q):=\left\{\begin{matrix}1 & {\text { if $s(p)=t(q)$}} \\ 0 & {\text { otherwise}}
\end{matrix}\right..
$$
Observe that $E(p,q)=E\bigl(s(p),t(q)\bigr)$.  If we write $B$ to be the subalgebra of $M(A)$ 
given by the pull-back of the algebra $K(G^{(0)})$ via the source map $s:G\to G^{(0)}$, and 
$C\subseteq M(A)$ the pull-back of the algebra $K(G^{(0)})$ via the target map $t:G\to G^{(0)}$, 
we will have $E\in M(B\otimes C)\subseteq M(A\otimes A)$.  It turns out that $E\in M(B\otimes C)$ 
satisfies the conditions to be regarded as what is known as a {\em separability idempotent\/}. 
See \cite{BJKVD_SepId}.

Meanwhile, even when $A$ is unital, we have $\Delta(1)=E\ne1\otimes1$, unless $G$ is a group.  
This means that the comultiplication map $\Delta$ cannot be a non-degenerate map.  This is 
significant, because in the quantum group theory, the non-degeneracy of $\Delta:A\to M(A\otimes A)$ 
allows one to extend $\Delta$ to the multiplier level.  Without the non-degeneracy of $\Delta$, 
so no straightforward mechanism to extend $\Delta$ to $M(A)$, it is not obvious how to make 
sense of the coassociativity condition: $(\Delta\otimes\operatorname{id})\Delta
=(\operatorname{id}\otimes\Delta)\Delta$.  Fortunately, using the existence of $E$, one can 
show that $\Delta$ does extend to $M(A)$.  This is how the notion of a {\em weak multiplier 
Hopf algebra\/} is formulated \cite{VDWangwha0}, \cite{VDWangwha1}.

In this paper, by requiring the existence of such an idempotent $E$ in the $C^*$-algebra 
framework and by axiomatizing its properties, we aim to develop a notion of a locally compact 
quantum groupoid that would be a $C^*$-algebraic analogue of a weak multiplier Hopf algebra. 

The theory is further developed in Part~II \cite{BJKVD_qgroupoid2}, where the left/right 
regular representations and the antipode map are constructed and the quasi-invariance 
of $\nu$ is proved.  The resulting class of locally compact quantum groupoids, given by 
the data $(A,\Delta,E,B,\nu,\varphi,\psi)$, contains as a special case all locally compact 
quantum groups (when $B=\mathbb{C}$ and $E=1\otimes1$).  Later (in Part III in our series 
\cite{BJKVD_qgroupoid3}), we aim to show that this class forms a self-dual category, satisfying 
a generalized Pontryagin-type duality.

However, in general, one cannot always expect such an idempotent element $E$ to exist as 
a bounded operator.  The existence of $E$ restricts the possible choice of $(B,\nu)$, which 
means that our class of locally compact quantum groups would not be fully general.

In the purely algebraic setting, a more general theory is currently being developed, namely, the notion 
of {\em multiplier Hopf algebroids}, by Timmermann and the second-named author (Van Daele) 
\cite{TimVD}.  Meanwhile, in the von Neumann algebraic setting, there exists a closely related 
notion of {\em measured quantum groupoids\/}, by Lesieur \cite{LesSMF} and Enock \cite{EnSMF}.

In the measured quantum groupoid approach, the comultiplication is defined as a map from 
a von Neumann algebra $M$ into $M\,{\substack{\ast\\N}}\,M$, a ``fiber product'' of $M$ with 
itself.  Unfortunately, the notion of a fiber product of von Neumann algebras (over a ``relative 
tensor product'' of Hilbert spaces) is highly technical, and moreover, there is no direct analogue 
of such a notion in the $C^*$-algebra framework.

Such being the current status of the research, the authors feel that it is still quite worthwhile 
to study this (sub)class of $C^*$-algebraic locally compact quantum groupoids based on 
the notion of weak multiplier Hopf algebras, with the comultiplication defined as map from 
$A$ into $M(A\otimes A)$.  While restrictive,  this is precisely the reason our framework 
can be more accessible than the rather technical approaches like the measured quantum 
groupoids or the Hopf algebroids.  In addition, the theory itself is rich in nature, being 
a self-dual category that naturally extends the theory of locally compact quantum groups. 
These aspects will be further developed in the upcoming papers in the series (\cite{BJKVD_qgroupoid2}, 
\cite{BJKVD_qgroupoid3}). 

Let us also point out the fact that until now, there has not been any $C^*$-algebraic theory developed 
for non-unital (so non-compact) cases.  This is done in the present paper.  Our framework includes 
some recent nontrivial examples (for instance, \cite{DeComm_LinkingQG}, \cite{DCTT_partialCQG}, 
\cite{DC_C*partialCQG}), and we expect to construct other examples in the upcoming years.
As the more general $C^*$-algebraic theory based on Hopf algebroids is still some time away 
from being developed, our proposed framework here will provide us with a nice bridge.

The current paper is organized as follows.  In Section~1, we gather some basic results concerning 
$C^*$-algebra weights.  The purpose here is to set the terminologies and notations to be used 
in later sections as well as in Parts~II and III.  We skip most of the proofs, giving references 
instead.

In Section~2, we review the notion of a ``separability idempotent'' in the $C^*$-algebra setting. 
As noted above, the existence of such an element plays an important role in our approach.  More 
details can be found in a separate paper by the authors \cite{BJKVD_SepId}.

After these preliminaries, definitions and immediate properties of the comultiplication $\Delta$ 
and the canonical idempotent $E$ are given in Section~3.  In particular, we show how the existence 
of $E$ allows us to extend the comultiplication map to the multiplier algebra level, thereby enabling 
us to make better sense of the coassociativity condition.

Definitions for the left and the right Haar weights are given in Section~4.  We then give the main 
definition of our locally compact quantum groupoid.  In Section~5, we give some examples. 
As the proper theory is yet to be developed (in Parts~II and III), we included only the basic 
examples here.

The series continues in Part~II \cite{BJKVD_qgroupoid2}.  In that paper, we construct from the 
defining axioms the partial isometries $V$ and $W$, which are essentially the right and the left 
regular representations.  We will also construct the antipode map $S$ and its polar decomposition. 
Later in Part~III \cite{BJKVD_qgroupoid3}, we will consider the dual object, which is also a locally 
compact quantum groupoid of our type.  We will then obtain the Pontryagin-type duality result.

\bigskip

{\sc Acknowledgments.}

The current work began during the first named author (Kahng)'s sabbatical leave visit 
to the University of Leuven during 2012/2013.  He is very much grateful to his coauthor 
(Van Daele) and the mathematics department at University of Leuven for their warm support 
and hospitality during his stay.   He also wishes to thank Michel Enock for the hospitality and 
support he received while he was visiting Jussieu.  The discussions on measured quantum 
groupoids, together with his encouraging words, were all very inspiring and helpful. 

\bigskip

\section{Preliminaries on $C^*$-algebra weights}

In this section, we review some basic notations and results concerning the weights on 
$C^*$-algebras, which will be useful later.  For standard terminologies and for a more 
complete treatment on $C^*$-algebra weights, refer to \cite{Cm}, \cite{Str}, \cite{Tk2}. 
There is also a nice survey \cite{KuVaweightC*}, and one may also refer to section~1 of 
\cite{KuVa}.

\subsection{Proper weights}

Let $A$ be a $C^*$-algebra.  A function $\psi:A^+\to[0,\infty]$ is called a weight on $A$, 
if $\psi(x+y)=\psi(x)+\psi(y)$ for all $x,y\in A^+$ and $\psi(\lambda x)=\lambda\psi(x)$ 
for all $x\in A^+$ and $\lambda\in\mathbb{R}^+$, with the convention that $0\cdot\infty=0$.

Given a weight $\psi$ on $A$, we can consider the following subspaces:
\begin{align}
{\mathfrak M}_\psi^+:&=\bigl\{a\in A^+:\psi(a)<\infty\bigr\} \notag \\
{\mathfrak N}_\psi:&=\bigl\{a\in A:\psi(a^*a)<\infty\bigr\} \notag \\
{\mathfrak M}_\psi:&={\mathfrak N}_\psi^*{\mathfrak N}_\psi
=\operatorname{span}\{y^*x:x,y\in{\mathfrak N}_\psi\} \notag
\end{align}
It is easy to see that ${\mathfrak N}_\psi$ is a left ideal in $M(A)$, and that 
${\mathfrak M}_\psi\subseteq{\mathfrak N}_\psi$.  Moreover, ${\mathfrak M}_\psi$ 
is a ${}^*$-subalgebra of $A$ spanned by ${\mathfrak M}_\psi^+$, which is its positive 
part.  Extending $\psi$ on ${\mathfrak M}_\psi^+$, we naturally obtain a map from 
${\mathfrak M}_\psi$ into $\mathbb{C}$, which we will still denote by $\psi$.

A weight $\psi$ is said to be faithful, if $\psi(a)=0$, $a\in A^+$, implies $a=0$.  We say 
a weight $\psi$ is densely-defined (or semi-finite), if ${\mathfrak N}_\psi$ is dense in $A$. 
The weights we will be considering below are ``proper weights'': A proper weight on a 
$C^*$-algebra is a non-zero, densely-defined weight, which is also lower semi-continuous. 

It is useful to consider the following sets, first considered by Combes \cite{Cm}:
\begin{align}
{\mathcal F}_\psi&:=\bigl\{\omega\in A^*_+:\omega(x)\le\psi(x),\forall x\in A^+\bigr\}, \notag \\
{\mathcal G}_\psi&:=\bigl\{\alpha\omega:\omega\in{\mathcal F}_\psi,{\text { for $\alpha\in(0,1)$}}\bigr\}.
\notag
\end{align}
Here $A^*$ denotes the norm dual of $A$.  Note that on ${\mathcal F}_\psi$, one can give a natural 
order inherited from $A^*_+$.  Then ${\mathcal G}_\psi$ is a directed subset of ${\mathcal F}_\psi$. 
Because of this, ${\mathcal G}_\psi$ is often used as an index set (of a net).  For a proper weight, 
being lower semi-continuous, we would have: $\psi(x)=\lim_{\omega\in{\mathcal G}_\psi}
\bigl(\omega(x)\bigr)$, for $x\in A^+$.

One can associate to $\psi$ a GNS-construction $({\mathcal H}_\psi,\pi_\psi,\Lambda_\psi)$. Here, 
${\mathcal H}_\psi$ is a Hilbert space, $\Lambda_\psi:{\mathfrak N}_\psi\to{\mathcal H}_\psi$ is 
a linear map such that $\Lambda_\psi({\mathfrak N}_\psi)$ is dense in ${\mathcal H}_\psi$, and 
$\bigl\langle\Lambda_\psi(a),\Lambda_\psi(b)\bigr\rangle=\psi(b^*a)$ for $a,b\in{\mathfrak N}_\psi$. 
Since $\psi$ is assumed to be lower semi-continuous, we further have that $\Lambda_\psi$ is closed. 
And, $\pi_\psi$ is the GNS representation of $A$ on ${\mathcal H}_\psi$, given by $\pi_\psi(a)
\Lambda_\psi(b)=\Lambda_\psi(ab)$ for $a\in A$, $b\in{\mathfrak N}_\psi$.  The GNS representation 
is non-degenerate, and the GNS construction is unique up to a unitary transformation.

Every $\omega\in A^*$ has a unique extension to the level of the multiplier algebra $M(A)$, which 
we may still denote by $\omega$.  From this fact, it follows easily that any proper weight $\psi$ 
on $A$ has a natural extension to the weight $\bar{\psi}$ on $M(A)$.  For convenience, we will use 
the notations $\overline{\mathfrak N}_\psi={\mathfrak N}_{\bar{\psi}}$ and $\overline{\mathfrak M}_\psi
={\mathfrak M}_{\bar{\psi}}$.  Also, the GNS construction for a proper weight $\psi$ on $A$ has 
a natural extension to the GNS construction for $\bar{\psi}$ on $M(A)$, with ${\mathcal H}_{\bar{\psi}}
={\mathcal H}_\psi$.  Often, we will just write $\psi$ for $\bar{\psi}$.

\subsection{KMS weights}

Of particular interest is the notion of KMS weights (see \cite{Cm}, and Chapter~VIII of \cite{Tk2}). 
The notion as defined below (due to Kustermans \cite{KuKMS}) is slightly different from the original 
one given by Combes, but equivalent.

\begin{defn}\label{KMSweight}
Let $\psi$ be a faithful, semi-finite, lower semi-continuous weight.  It is called a {\em KMS weight\/}, 
if there exists a norm-continuous one-parameter group of automorphisms $(\sigma_t)_{t\in\mathbb{R}}$ 
of $A$ such that $\psi\circ\sigma_t=\psi$ for all $t\in\mathbb{R}$, and 
$$
\psi(a^*a)=\psi\bigl(\sigma_{i/2}(a)\sigma_{i/2}(a)^*\bigr) {\text { for all $a\in{\mathcal D}(\sigma_{i/2})$.}}
$$
Here, $\sigma_{i/2}$ is the analytic generator of the one-parameter group $(\sigma_t)$ at $z=i/2$, 
and ${\mathcal D}(\sigma_{i/2})$ is its domain.  In general, with the domain properly defined, it is 
known that $\sigma_z$, $z\in\mathbb{C}$, is a closed, densely-defined map.
\end{defn}

The one-parameter group $(\sigma_t)$ is called the ``modular automorphism group'' for $\psi$. 
It is uniquely determined because $\psi$ is faithful.  In the special case when $\psi$ is a trace, 
that is, $\psi(a^*a)=\psi(aa^*)$ for $a\in{\mathfrak N}_\psi$, it is clear that $\psi$ is KMS, 
with the modular automorphism group being trivial ($\sigma\equiv\operatorname{Id}$).  In a sense, 
the KMS condition provides somewhat of a control over the non-commutativity of $A$.

Basic properties of KMS weights can be found in \cite{Tk2}, \cite{KuKMS}, \cite{KuVaweightC*}. 
In particular, there exists a (unique) anti-unitary operator $J$ (the ``modular conjugation'') 
on ${\mathcal H}_\psi$ such that $J\Lambda_\psi(x)=\Lambda_\psi\bigl(\sigma_{i/2}(x)^*\bigr)$, 
for $x\in{\mathfrak N}_\psi\cap{\mathcal D}(\sigma_{i/2})$.  There exists also a strictly positive 
operator $\nabla$ (the ``modular operator'') on ${\mathcal H}_\psi$ such that 
$\nabla^{it}\Lambda_\psi(a)=\Lambda_\psi\bigl(\sigma_t(a)\bigr)$, for $a\in{\mathfrak N}_{\psi}$ 
and $t\in\mathbb{R}$.  See below for a standard result, which can be easily extended to 
elements in the multiplier algebra:

\begin{lem}\label{lemKMS}
Let $\psi$ be a KMS weight on a $C^*$-algebra $A$, with GNS representation $({\mathcal H}_\psi,
\pi_\psi,\Lambda_\psi)$.  Then we have:
\begin{enumerate}
  \item Let $a\in{\mathcal D}(\sigma_{i/2})$ and $x\in{\mathfrak N}_\psi$.  Then $xa\in{\mathfrak N}_\psi$ 
and $\Lambda_\psi(xa)=J\pi_\psi\bigl(\sigma_{i/2}(a)\bigr)^*J\Lambda_\psi(x)$.
  \item Let $a\in{\mathcal D}(\sigma_{-i})$ and $x\in{\mathfrak M}_\psi$.  Then $ax, x\sigma_{-i}(a)
\in{\mathfrak M}_\psi$ and $\psi(ax)=\psi\bigl(x\sigma_{-i}(a)\bigr)$.
  \item Let $x\in{\mathfrak N}_\psi\cap{\mathfrak N}_\psi^*$ and $a\in{\mathfrak N}_\psi^*\cap 
{\mathcal D}(\sigma_{-i})$ be such that $\sigma_{-i}(a)\in{\mathfrak N}_\psi$.  Then $\psi(ax)
=\psi\bigl(x\sigma_{-i}(a)\bigr)$.
\end{enumerate}
\end{lem}

\begin{rem}
It is known that when we lift a KMS weight to the level of von Neumann algebra $\pi_\psi(A)''$ 
in an evident way, we obtain a normal, semi-finite, faithful (``n.s.f.'')~weight $\widetilde{\psi}$. 
See \cite{Tk2}.  In addition, the modular automorphism group $(\sigma^{\widetilde{\psi}}_t)$ for 
the n.s.f.~weight $\widetilde{\psi}$ leaves the $C^*$-algebra $B$ invariant, and the restriction 
of $(\sigma^{\widetilde{\psi}}_t)$ to $B$ coincides with our $(\sigma_t)$.  The operators $J$, 
$\nabla$  above are none other than the restrictions of the corresponding operators arising 
in the standard Tomita--Takesaki theory of von Neumann algebra weights.
\end{rem}

In Tomita--Takesaki theory, a useful role is played by the ``Tomita ${}^*$-algebra''.  In our case, 
consider the n.s.f.~weight $\tilde{\psi}$ on the von Neumann algebra $\pi_\psi(A)''$.  Then the 
Tomita ${}^*$-algebra for $\tilde{\psi}$, denoted by ${\mathcal T}_{\tilde{\psi}}$, is as follows:
$$
{\mathcal T}_{\tilde{\psi}}:=\bigl\{x\in{\mathfrak N}_{\tilde{\psi}}\cap{\mathfrak N}_{\tilde{\psi}}^*:
{\text { $x$ is analytic w.r.t. $\sigma^{\tilde{\psi}}$ and $\sigma^{\tilde{\psi}}_z(x)
\in{\mathfrak N}_{\tilde{\psi}}\cap{\mathfrak N}_{\tilde{\psi}}^*$, $\forall z\in\mathbb{C}$}}\bigr\}.
$$
It is a strongly ${}^*$-dense subalgebra in $\pi_\psi(A)''$.  Refer to the standard textbooks 
on the modular theory \cite{Tk2}, \cite{Str}.  In particular, for any $z\in\mathbb{C}$, we have 
$\sigma^{\tilde{\psi}}_z({\mathcal T}_{\tilde{\psi}})={\mathcal T}_{\tilde{\psi}}$. So, the 
$\sigma^{\tilde{\psi}}_z$ are densely-defined, as well as have dense ranges.  

It is also known that ${\mathcal T}_{\psi}:={\mathcal T}_{\tilde{\psi}}\cap A$ is (norm)-dense 
in $A$, so turns out to be useful as we work with the weight $\psi$ at the $C^*$-algebra level. 
For the $\sigma^{\psi}_z$, $z\in\mathbb{C}$, we have $\sigma^{\psi}_z({\mathcal T}_{\psi})
=\sigma^{\psi}_z({\mathcal T}_{\tilde{\psi}}\cap A)={\mathcal T}_{\tilde{\psi}}\cap A
={\mathcal T}_{\psi}$.  So, as above, the $\sigma^{\psi}_z$ are densely-defined, 
as well as have dense ranges.  

\subsection{Central weights (approximate KMS weights)}

A proper weight $\psi$ on a $C^*$-algebra $A$ is said to be a ``central weight'' \cite{VDvN}, if 
the support projection of its extension weight $\tilde{\psi}$ to the level of the double dual $A^{**}$ 
is a central projection in $A^{**}$.  This notion turns out to be equivalent to that of an 
``approximate KMS weight'', considered by Kustermans and Vaes \cite{KuVa} in their theory of 
locally compact quantum groups.  While being a faithful central weight is a weaker 
condition than being a KMS weight, $\psi$ remains faithful when extended to $\pi_{\psi}(A)''$, 
giving us an n.s.f. weight.

We intend to work primarily with the more familiar notion of KMS weights in our series of papers. 
Yet, it is worth noting that often, we can weaken the requirements to work instead with the central 
weights.  We will consider some of these questions in Part~III.

\subsection{Slicing with weights}

Let $A$ and $B$ be two $C^*$-algebras, with a proper weight $\psi$ on $A$ and its associated GNS-triple 
$({\mathcal H}_{\psi},\pi_{\psi},\Lambda_{\psi})$.  One can make sense of the slice map $\psi\otimes
\operatorname{id}$, which would be an operator-valued weight.  Of course, it is also possible to consider 
the slice from the other side.  Here we collect some notations and results, without proof.  See section~3 
of \cite{KuVaweightC*}.

For $x\in M(A\otimes B)^+$, we say $x\in\overline{\mathfrak M}_{\psi\otimes\operatorname{id}}^+$, if 
the net $\bigl((\theta\otimes\operatorname{id})(x)\bigr)_{\theta\in{\mathcal G}_{\psi}}$ is strictly 
convergent in $B$.  So, for $x\in\overline{\mathfrak M}_{\psi\otimes\operatorname{id}}^+$, one can 
define $(\psi\otimes\operatorname{id})(x)\in M(B)$.  Analogously as before, one can extend 
$\psi\otimes\operatorname{id}$ to $\overline{\mathfrak M}_{\psi\otimes\operatorname{id}}$, the linear 
span of $\overline{\mathfrak M}_{\psi\otimes\operatorname{id}}^+$ in $M(A\otimes B)$.  It is 
a ${}^*$-subalgebra of $M(A\otimes B)$.  We can also consider $\overline{\mathfrak N}_{\psi\otimes
\operatorname{id}}:=\bigl\{x\in M(A\otimes B):x^*x\in\bar{\mathfrak M}_{\psi\otimes
\operatorname{id}}^+\bigr\}$, which is a left ideal in $M(A\otimes B)$.

Here are some results:

\begin{prop}\label{slice_lem1}
Consider $x\in\overline{\mathfrak M}_{\psi\otimes\operatorname{id}}$ and $\omega\in B^*$.  Then 
$(\operatorname{id}\otimes\omega)(x)\in\overline{\mathfrak M}_{\psi}$, and we have: 
$\psi\bigl((\operatorname{id}\otimes\omega)(x)\bigr)=\omega\bigl((\psi\otimes\operatorname{id})(x)\bigr)$.
\end{prop}

\begin{prop}
Let $x\in M(A\otimes B)^+$ and $b\in M(B)^+$ be such that $(\operatorname{id}\otimes\omega)(x)
\in\overline{\mathfrak M}_{\psi}^+$ and that $\psi\bigl((\operatorname{id}\otimes\omega)(x)\bigr)
=\omega(b)$ for all $\omega\in B^*_+$.  Then $x\in\overline{\mathfrak M}_{\psi\otimes\operatorname{id}}^+$, 
and we have: $(\psi\otimes\operatorname{id})(x)=b$.
\end{prop}

\begin{prop}\label{cauchyschwarz}
Let $x,y\in\overline{\mathfrak N}_{\psi\otimes\operatorname{id}}$.  Then the following ``generalized 
Cauchy--Schwarz inequality'' holds:
$$
\bigl((\psi\otimes\operatorname{id})(y^*x)\bigr)^*\bigl((\psi\otimes\operatorname{id})(y^*x)\bigr)
\le\bigl\|(\psi\otimes\operatorname{id})(y^*y)\bigr\|(\psi\otimes\operatorname{id})(x^*x).
$$
\end{prop}

\subsection{tensor product of weights}

Let $A$ and $B$ be two $C^*$-algebras.  Consider a proper weight $\psi$ on $A$ and a proper weight 
$\varphi$ on $B$.  One may wish to consider the notion of the tensor product weight $\psi\otimes\varphi$ 
on $A\otimes B$.  Here again, we collect below some notations and results, without proof.  See 
section~4 of \cite{KuVaweightC*}. 

One can define the tensor product weight $\psi\otimes\varphi$ on $A\otimes B$, as follows:
$$
(\psi\otimes\varphi)(x)=\sup\bigl\{(\theta\otimes\omega)(x)|\theta\in{\mathcal F}_{\psi},
\omega\in{\mathcal F}_{\varphi}\bigr\},\quad {\text { for $x\in(A\otimes B)^+$.}}
$$
This definition can be naturally extended to the multiplier algebra level.

By definition, it is easy to see the following Fubini-type result:
$$
\varphi\bigl((\psi\otimes\operatorname{id})(x)\bigr)=(\psi\otimes\varphi)(x),
$$
for $x\in{\mathfrak M}_{\psi\otimes\operatorname{id}}^+$.  Obviously, the roles of $\psi$ and 
$\varphi$ can be reversed.

Fix GNS-triples $({\mathcal H}_{\psi},\pi_{\psi},\Lambda_{\psi})$ for $\psi$ and 
$({\mathcal H}_{\varphi},\pi_{\varphi},\Lambda_{\varphi})$ for $\varphi$.  In general, the space 
${\mathcal H}_{\psi}\otimes{\mathcal H}_{\varphi}$ could be too small to serve as a GNS-space 
for $\psi\otimes\varphi$.  So one introduces a special subset $\overline{\mathfrak N}
(\psi,\varphi)\,\bigl(\subseteq\overline{\mathfrak N}_{\psi\otimes\varphi}\bigr)$, on which 
the GNS map $\Lambda_{\psi}\otimes\Lambda_{\varphi}:\overline{\mathfrak N}(\psi,\varphi)
\to{\mathcal H}_{\psi}\otimes{\mathcal H}_{\varphi}$ can be defined.  To be a little more precise, 
we have $x\in\overline{\mathfrak N}(\psi,\varphi)$, if there exists a $v\in{\mathcal H}_{\psi}
\otimes{\mathcal H}_{\varphi}$ such that $\|v\|^2=(\psi\otimes\varphi)(x^*x)$ and $\bigl\langle
v,\Lambda_{\psi}(a)\otimes\Lambda_{\varphi}(b)\bigr\rangle=(\psi\otimes\varphi)\bigl((a^*\otimes b^*)
x\bigr)$ for $a\in{\mathfrak N}_{\psi}$, $b\in{\mathfrak N}_{\varphi}$.

If one of the weights $\psi$ and $\varphi$ is a KMS weight, then it turns out that 
$\overline{\mathfrak N}(\psi,\varphi)=\bar{\mathfrak N}_{\psi\otimes\varphi}$, and 
$({\mathcal H}_{\psi}\otimes{\mathcal H}_{\varphi},\pi_{\psi}\otimes\pi_{\varphi},
\Lambda_{\psi}\otimes\Lambda_{\varphi})$ is indeed a GNS-construction for $\phi\otimes\varphi$. 

In particular, if both $\psi$ and $\varphi$ are KMS weights, then ${\mathfrak N}_{\psi}\otimes
{\mathfrak N}_{\varphi}\,(\subseteq{\mathfrak N}_{\psi\otimes\varphi})$ is a core for the 
GNS-map $\Lambda_{\psi}\otimes\Lambda_{\varphi}$, and $\psi\otimes\varphi$ also becomes 
a KMS weight on $A\otimes B$.  Its modular automorphism group is $\sigma^{\psi}\otimes\sigma^{\varphi}$.

\section{Separability idempotents}

The quantum groupoids that we wish to consider in this paper are modeled after the (purely algebraic) 
notion of {\em weak multiplier Hopf algebras\/} \cite{VDWangwha0}, \cite{VDWangwha1}, which generalizes 
the {\em weak Hopf algebras\/} \cite{BNSwha1}, \cite{BSwha2}.  In these approaches, the existence 
of a certain idempotent element $E$ plays a central role.  Among its properties is the fact that it is a 
{\em separability idempotent\/}.  So for our purposes, we need to define and clarify such a notion in the 
setting of $C^*$-algebras.  This was done in the authors' previous paper \cite{BJKVD_SepId}.  In this 
section, we give a brief summary.  Some of the results are provided without proof, and some only with a 
lighter version of the proof, with details referred to the previous paper.

We begin by considering a triple $(E,B,\nu)$, where
\begin{itemize}
  \item $B$ is a $C^*$-algebra;
  \item $\nu$ is a faithful KMS weight on $B$, together with its associated norm-continuous automorphism 
        group $(\sigma^{\nu}_t)_{t\in\mathbb{R}}$;
  \item $E\in M(B\otimes C)$ is a self-adjoint idempotent element for some $C^*$-algebra $C$ such that 
        there exists a ${}^*$-anti-isomorphism $R_{BC}:B\to C$.
\end{itemize}

In the above, it is clear that $C\cong B^{\operatorname{op}}$ as $C^*$-algebras.  Even so, there is 
no reason to assume that $C$ and $B^{\operatorname{op}}$ are exactly same.  Later, $B$ and $C$ 
will be the ``source algebra'' and the ``target algebra'', respectively, with the weight $\nu$ being analogous 
to choosing a suitable measure on the unit space of a locally compact groupoid.   Going forward, we will 
understand our triple $(E,B,\nu)$ in such a way that the $C^*$-algebra $C$ and the ${}^*$-aniti-isomorphism 
$R_{BC}:B\to C$ are implicitly fixed, with $C$ possibly different from $B^{\operatorname{op}}$.  Moreover, 
for convenience, unless we really need to specify the algebras $B$ and $C$, we will just write $R$ 
to denote the map $R_{BC}$. 

We give below the definition of a separability idempotent. See Definition~2.2 of \cite{BJKVD_SepId}.

\begin{defn}\label{separabilitytriple}
We will say that $(E,B,\nu)$ is a {\em separability triple\/}, if the following conditions hold:
\begin{enumerate}
  \item $(\nu\otimes\operatorname{id})(E)=1$
  \item For $b\in{\mathcal D}(\sigma^{\nu}_{i/2})$, we have: $(\nu\otimes\operatorname{id})
\bigl(E(b\otimes1)\bigr)=R\bigl(\sigma^{\nu}_{i/2}(b)\bigr)$.
\end{enumerate}
If $(E,B,\nu)$ forms a separability triple, then we say $E$ is a {\em separability idempotent\/}.
\end{defn}

\begin{rem}
In the definition above, (1) means that for any $\omega\in C^*_+$, 
we require that $(\operatorname{id}\otimes\omega)(E)\in\overline{\mathfrak M}_\nu\,\bigl(\subseteq 
M(B)\bigr)$ and that $\nu\bigl((\operatorname{id}\otimes\omega)(E)\bigr)=\omega(1)$.  From this, 
it will follow that $(\operatorname{id}\otimes\omega)(E)\in\overline{\mathfrak M}_\nu$ for any 
$\omega\in C^*$, and that $\nu\bigl((\operatorname{id}\otimes\omega)(E)\bigr)=\omega(1),
\forall\omega\in C^*$.  Meanwhile, by Lemma~\ref{lemKMS}, having $(\operatorname{id}\otimes\omega)(E)
\in\overline{\mathfrak M}_\nu$, $\omega\in C^*$,  means $(\operatorname{id}\otimes\omega)(E)
b\in{\mathfrak M}_\nu$, for $b\in{\mathcal D}(\sigma^{\nu}_{i/2})$.  Condition~(2) means that 
when we apply $\nu$ here, we require that $\nu\bigl((\operatorname{id}\otimes\omega)(E)b\bigr)
=\omega\bigl((R\circ\sigma^{\nu}_{i/2})(b)\bigr)$, $\forall\omega$.  In other words,  $(\nu\otimes\operatorname{id})
\bigl(E(b\otimes1)\bigr)=R\bigl(\sigma^{\nu}_{i/2}(b)\bigr)$.
\end{rem}

\begin{prop}\label{E_unique}
Let $E$ be a separability idempotent.  Then $E$ is uniquely determined by the data $(B,\nu,R)$.
\end{prop}

\begin{proof}
See Proposition~2.3 of \cite{BJKVD_SepId}.
\end{proof}

As we regard a separability triple $(E,B,\nu)$ to implicitly fix the $C^*$-algebra $C$ and the anti-isomorphism 
$R$, this result means that $E$ is more or less characterized by the pair $(B,\nu)$. However, there is no 
reason to believe that $E$ determined in this way would have to be a projection.  In other words, the 
existence of a separability idempotent $E$ is really a condition on the pair $(B,\nu)$.  See \cite{BJKVD_SepId} 
for a more careful discussion.  In particular, in section~6 of that paper, it is shown that $B$ has to be postliminal.

Let us now collect some properties of a separability idempotent.  We first begin with the existence 
of a densely-defined map $\gamma_B:B\to C$.

\begin{prop}\label{gammamap}
\begin{enumerate}
  \item For $b\in{\mathcal D}(\sigma^{\nu}_{i/2})$, write $\gamma_B(b):=R\bigl(\sigma^{\nu}_{i/2}(b)\bigr)\in C$. 
This determines $\gamma_B$ as a closed, densely-defined map from $B$ to $C$.  It is also injective, 
and has a dense range in $C$.
  \item $\gamma_B(b)=(\nu\otimes\operatorname{id})\bigl(E(b\otimes1)\bigr)$, for $b\in{\mathcal D}(\gamma_B)
={\mathcal D}(\sigma^{\nu}_{i/2})$.
  \item ${\mathcal D}(\gamma_B)$ is closed under multiplication, and the map $\gamma_B$ is an anti-homomorphism.
\end{enumerate}
\end{prop}

\begin{proof}
(1). The injectivity is clear, since $R$ is an anti-isomorphism and $\sigma^{\nu}_t$ is an automorphism for all $t$. 
Since ${\mathcal D}(\sigma^{\nu}_{i/2})$ is dense in $B$, and since $\sigma^{\nu}_{i/2}$ is a closed map, 
we can see that $\gamma_B$ is a closed, densely-defined map from $B$ to $C$.  Meanwhile, by the property 
of the Tomita algebra, we know that ${\mathcal T}_{\nu}\subseteq
{\mathcal D}(\sigma^{\nu}_{i/2})$ and that $\sigma^{\nu}_{i/2}({\mathcal T}_{\nu})
={\mathcal T}_{\nu}$, which is dense in $B$ (see Section~1).  Since $R$ is an anti-isomorphism between $B$ and $C$, 
we see that $\operatorname{Ran}(\gamma_B)\supseteq (R\circ\sigma^{\nu}_{i/2})({\mathcal T}_{\nu})
=R({\mathcal T}_{\nu})$, which is dense in $C$.

(2). This is a re-formulation of Definition~\ref{separabilitytriple}\,(2).

(3). Suppose $b,b'\in{\mathcal D}(\sigma^{\nu}_{i/2})$.  They are analytic elements, and we have 
$bb'\in{\mathcal D}(\sigma^{\nu}_{i/2})$. It is clear that $\gamma_B$ is an anti-homomorphism, because 
$\sigma^{\nu}_{i/2}$ and $R$ an anti-isomorphism: $\gamma_B(bb')=\gamma_B(b')\gamma(b)$.
\end{proof}

Consider $\operatorname{Ran}(\gamma_B)$.  By Proposition~\ref{gammamap}\,(1), it is dense in $C$. 
In view of the results above about $\gamma_B$, it is clear that we can consider $\gamma_B^{-1}:
\operatorname{Ran}(\gamma_B)\to B$, which would be a closed, densely-defined map from $C$ 
to $B$, having dense range in $B$, injective, and anti-multiplicative.  In fact, we will have: 
$\gamma_B^{-1}=\sigma^{\nu}_{-i/2}\circ R^{-1}$, and ${\mathcal D}(\gamma_B^{-1})
=\operatorname{Ran}(\gamma_B)$, while $\operatorname{Ran}(\gamma_B^{-1})
={\mathcal D}(\gamma_B)$.

In the below is a property that relates the $C^*$-algebras $B$ and $C$, at the level of their dense subspaces 
${\mathcal D}(\gamma_B)\subseteq B$ and $\operatorname{Ran}(\gamma_B)\subseteq C$.

\begin{prop}\label{sepidgamma}
\begin{enumerate}
  \item For $b\in{\mathcal D}(\gamma_B)$, we have $E(b\otimes1)=E\bigl(1\otimes\gamma_B(b)\bigr)$.
  \item For $c\in\operatorname{Ran}(\gamma_B)={\mathcal D}(\gamma_B^{-1})$, we have $E(1\otimes c)
=E\bigl(\gamma_B^{-1}(c)\otimes1\bigr)$.
\end{enumerate}
\end{prop}

\begin{proof}
This is Proposition~2.6 of \cite{BJKVD_SepId}. Note that for $b,b'\in{\mathcal D}(\gamma_B)$, 
\begin{align}
\bigl(\nu(\cdot b')\otimes\operatorname{id}\bigr)\bigl(E(b\otimes1)\bigr)&=
(\nu\otimes\operatorname{id})\bigl(E(bb'\otimes1)\bigr)
=\gamma_B(bb')=\gamma_B(b')\gamma_B(b) \notag \\
&=(\nu\otimes\operatorname{id})\bigl(E(b'\otimes1)\bigr)\gamma_B(b)
=(\nu\otimes\operatorname{id})\bigl(E(b'\otimes\gamma_B(b))\bigr)  \notag \\
&=(\nu\otimes\operatorname{id})\bigl(E(1\otimes\gamma_B(b))(b'\otimes1)\bigr) \notag \\
&=\bigl(\nu(\cdot b')\otimes\operatorname{id}\bigr)\bigl(E(1\otimes\gamma_B(b))\bigr).  \notag 
\end{align}
Since $\nu$ is faithful, and since the result is true for all $b'\in{\mathcal D}(\gamma_B)$, 
which is dense in $B$, we conclude that $E(b\otimes1)=E\bigl(1\otimes\gamma_B(b)\bigr)$.
Second result is an immediate consequence, where $c=\gamma_B(b)$, $\gamma_B^{-1}(c)=b$.
\end{proof}

From $(B,\nu)$, using the ${}^*$-anti-isomorphism $R:B\to C$, we can define a faithful weight 
$\mu$ on $C$, by $\mu:=\nu\circ R^{-1}$.  It is not difficult to show that $\mu$ is also 
a faithful KMS weight on the $C^*$-algebra $C$, together with the one-parameter group of 
automorphisms $(\sigma^{\mu}_t)_{t\in\mathbb{R}}$, given by 
$\sigma^{\mu}_t:=R\circ\sigma^{\nu}_{-t}\circ R^{-1}$. 
The pair $(C,\mu)$ would behave a lot like $(B,\nu)$, and we obtain from this 
a densely-defined map $\gamma_C:C\to B$.

\begin{prop}\label{gamma'map}
Let $\mu$ be as above.  Then we have:
\begin{enumerate}
  \item $(\operatorname{id}\otimes\mu)(E)=1$.
  \item For $c\in{\mathcal D}(\sigma^{\mu}_{-i/2})$, write: $\gamma_C(c):=(R^{-1}\circ\sigma^{\mu}_{-i/2})(c)$. 
This defines a closed, densely-defined map from $C$ to $B$.  It is also an injective anti-homomorphism, having 
a dense range in $B$.
  \item For $c\in{\mathcal D}(\gamma_C)={\mathcal D}(\sigma^{\mu}_{-i/2})$, we have: $(1\otimes c)E
=\bigl(\gamma_C(c)\otimes1)E$.  Also, for $b\in\operatorname{Ran}(\gamma_C)={\mathcal D}({\gamma_C}^{-1})$, 
we have: $(b\otimes 1)E=\bigl(1\otimes{\gamma_C}^{-1}(b)\bigr)E$. 
  \item For $c\in{\mathcal D}(\gamma_C)$, we have: $(\operatorname{id}\otimes\mu)
\bigl((1\otimes c)E\bigr)=\gamma_C(c)$.
\end{enumerate}
\end{prop}

\begin{proof}
These results are Propositions~3.1 and 2.7 in \cite{BJKVD_SepId}. Let us give a slightly different proof:

(1). As before, the equation means that $(\theta\otimes\operatorname{id})(E)\in\overline{\mathfrak M}_{\mu}$ 
for all $\theta\in B^*$ (extended to the multiplier algebra level), and that 
$\mu\bigl((\theta\otimes\operatorname{id})(E)\bigr)=\theta(1)$.  We can verify this for 
$\theta=\nu(\,\cdot\,b)$, for $b\in{\mathcal D}(\gamma_B)$. Such functionals are dense in $B^*$. Using the fact that
$\mu=\nu\circ R^{-1}$ and that $\nu$ is $\sigma^{\nu}$-invariant, we have: 
\begin{align}
\mu\bigl((\theta\otimes\operatorname{id})(E)\bigr)&=\mu\bigl((\nu\otimes\operatorname{id})
(E(b\otimes1))\bigr)=\mu\bigl(\gamma_B(b)\bigr)  \notag \\
&=\mu\bigl((R\circ\sigma^{\nu}_{i/2})(b)\bigr)=\nu(b)=\theta(1).
\notag
\end{align}

(2). From $(\nu\otimes\operatorname{id})\bigl(E(b\otimes1)\bigr)=R\bigl(\sigma^{\nu}_{i/2}(b)\bigr)$, 
$b\in{\mathcal D}(\gamma_B)={\mathcal D}(\sigma^{\nu}_{i/2})$, take the adjoint.  Since $R$ is a ${}^*$-anti-isomorphism, 
we have:
$$
(\nu\otimes\operatorname{id})\bigl((b^*\otimes1)E\bigr)=\bigl[R(\sigma^{\nu}_{i/2}(b))\bigr]^*
=R\bigl([\sigma^{\nu}_{i/2}(b)]^*\bigr)=R\bigl(\sigma^{\nu}_{-i/2}(b^*)\bigr).
$$
In other words, for $y\in{\mathcal D}(\gamma)^*={\mathcal D}(\sigma^{\nu}_{-i/2})$, which is also dense in $B$, 
the expression $(\nu\otimes\operatorname{id})\bigl((y\otimes1)E\bigr)$ is valid and 
$(\nu\otimes\operatorname{id})\bigl((y\otimes1)E\bigr)=(R\circ\sigma^{\nu}_{-i/2})(y)$. 
Or, put another way, we have: 
$$
\nu\bigl(y(\operatorname{id}\otimes\omega)(E)\bigr)
=\omega\bigr((R\circ\sigma^{\nu}_{-i/2})(y)\bigr),\quad{\text { for $\omega\in C^*$, 
$y\in D(\sigma^{\nu}_{-i/2})$.}}
$$
So, by the same argument as in the case of $\gamma_B$, the map 
$y\mapsto(R\circ\sigma^{\nu}_{-i/2})(y)$ is closed, densely-defined on $B$, injective, and has a dense 
range in $C$.  Let us define $\gamma_C$ to be its inverse map, namely, $\gamma_C:
c\to(\sigma^{\nu}_{i/2}\circ R^{-1})(c)=(R^{-1}\circ\sigma^{\mu}_{-i/2})(c)$. 
It is clear that $\gamma_C$ is closed, densely-defined in $C$, injective, has a dense range in $B$, 
as well as anti-multiplicative.

(3). The proof is essentially the same as in that of Proposition~\ref{sepidgamma}.  The anti-homorphism 
property of $\gamma_C$ is needed.

(4). Let $c\in D(\gamma_C)$.  Then using $(\operatorname{id}\otimes\mu)(E)=1$, we have:
$$
(\operatorname{id}\otimes\mu)\bigl((1\otimes c)E\bigr)=(\operatorname{id}\otimes\mu)\bigl((\gamma'(c)
\otimes 1)E\bigr)=\gamma'(c).
$$
\end{proof}

Here are some formulas relating the weights $\nu$, $\mu$, and the maps $\gamma_B$, $\gamma_C$.

\begin{prop}\label{gammagamma'formuals}
\begin{enumerate}
  \item We have:
\begin{alignat}{2}
&\gamma_B=R\circ\sigma^{\nu}_{i/2}=\sigma^{\mu}_{-i/2}\circ R &\qquad
&\gamma_B^{-1}=\sigma^{\nu}_{-i/2}\circ R^{-1}=R^{-1}\circ\sigma^{\mu}_{i/2} \notag \\
&\gamma_C=R^{-1}\circ\sigma^{\mu}_{-i/2}=\sigma^{\nu}_{i/2}\circ R^{-1} &
&\gamma_C^{-1}=\sigma^{\mu}_{i/2}\circ R=R\circ\sigma^{\nu}_{-i/2}  \notag
\end{alignat}
  \item $\mu=\nu\circ R^{-1}=\nu\circ\gamma_C=\nu\circ\gamma_B^{-1}$ and $\nu=\mu\circ R=\mu\circ\gamma_B=\mu\circ\gamma_C^{-1}$.
\end{enumerate}
\end{prop}

\begin{proof}
The definitions have already appeared, and we only need to use $\sigma^{\mu}_t
=R\circ\sigma^{\nu}_{-t}\circ R^{-1}$ for the alternative descriptions. For (2), use one of the characterizations 
in (1) together with the invariance of $\nu$ and $\mu$ under $\sigma^{\nu}$ and $\sigma^{\mu}$.
\end{proof}

The next result shows the relationship between the ${}^*$-structure and the maps $\gamma_B$, 
$\gamma_C$.  Observe that in general, the maps $\gamma_B$, $\gamma_C$ need not be ${}^*$-maps.

\begin{prop}\label{gamma*}
\begin{enumerate}
  \item For $b\in{\mathcal D}(\gamma_B)$, we have: $\gamma_C\bigl(\gamma_B(b)^*\bigr)^*=b$.
  \item For $c\in{\mathcal D}(\gamma_C)$, we have $\gamma_B\bigl(\gamma_C(c)^*\bigr)^*=c$.
\end{enumerate}
\end{prop}

\begin{proof}
This is Proposition~2.8 in \cite{BJKVD_SepId}. The proof is not difficult, using 
Propositions~\ref{sepidgamma} and \ref{gamma'map} above.
\end{proof}


Next proposition collects some results relating our separability idempotent $E$ with the algebras 
$B$ and $C$.  We skip the proof: One uses the properties of the maps $\gamma_B$, $\gamma_B^{-1}$, 
$\gamma_C$, $\gamma_C^{-1}$ obtained above, including the fact that they have dense domains and 
dense ranges.  For details, refer to \cite{BJKVD_SepId}.

\begin{prop}\label{Efull}
\begin{enumerate}
  \item For all $b\in B$ and all $c\in C$, we have:
$$
E(1\otimes c)\in B\otimes C,\quad (1\otimes c)E\in B\otimes C,\quad
(b\otimes 1)E\in B\otimes C,\quad E(b\otimes 1)\in B\otimes C.
$$
  \item As a consequence of (1), we have:
$$
(\operatorname{id}\otimes\omega)(E)\in B,\quad\forall\omega\in C^*,\quad
{\text { and }}\quad
(\theta\otimes\operatorname{id})(E)\in C,\quad\forall\theta\in B^*.
$$
  \item The separability idempotent $E$ is ``full'', in the sense that

$\bigl\{(\theta\otimes\operatorname{id})(E(b\otimes1)):b\in B,\theta\in B^*\bigr\}$ is dense in $C$,

$\bigl\{(\theta\otimes\operatorname{id})((b\otimes1)E):b\in B,\theta\in B^*\bigr\}$ is dense in $C$,

$\bigl\{(\operatorname{id}\otimes\omega)((1\otimes c)E):c\in C,\omega\in C^*\bigr\}$ is dense in $B$,

$\bigl\{(\operatorname{id}\otimes\omega)(E(1\otimes c)):c\in C,\omega\in C^*\bigr\}$ is dense in $B$.
  \item 
If $(1\otimes c)E=0$, $c\in C$, then necessarily $c=0$.

\noindent If $E(1\otimes c)=0$, $c\in C$, then necessarily $c=0$.

\noindent If $E(b\otimes 1)=0$, $b\in B$, then necessarily $b=0$.

\noindent If $(b\otimes 1)E=0$, $b\in B$, then necessarily $b=0$.
\end{enumerate}
\end{prop}

\begin{proof}
See Proposition~3.3, Proposition~3.4, Proposition~3.5 of \cite{BJKVD_SepId}.
\end{proof}

\begin{rem}
The ``fullness'' of $E\in M(B\otimes C)$, as given in (3) of Proposition~\ref{Efull}, means that the 
left leg of $E$ is $B$ and the right leg of $E$ is $C$.  In the purely algebraic setting, the fullness 
of $E$ was part of the definition of $E$ being a separability idempotent \cite{VDsepid}.  Here, 
by beginning instead with the weights $\mu$ and $\nu$, we obtain the fullness result as a proposition.
\end{rem}

In the next proposition, we see that the analytic generators $\sigma^{\mu}_{-i}$ and $\sigma^{\nu}_{-i}$ 
can be characterized in terms of the maps $\gamma_B$ and $\gamma_C$.

\begin{prop}\label{modularauto}
\begin{enumerate}
  \item $\sigma^{\mu}_{-i}(c)=(\gamma_B\circ\gamma_C)(c)$, for $c\in{\mathcal D}(\gamma_B\circ\gamma_C)$.
  \item $\sigma^{\nu}_{-i}(b)=(\gamma_B^{-1}\circ\gamma_C^{-1})(b)$, for $b\in{\mathcal D}(\gamma_B^{-1}\circ\gamma_C^{-1})$.  
\end{enumerate}
\end{prop}

\begin{proof}
For (1), use $\gamma_B=\sigma^{\mu}_{-i/2}\circ R$ and $\gamma_C=R^{-1}\circ\sigma^{\mu}_{-i/2}$, observed 
in Proposition~\ref{gammagamma'formuals}.  For (2), use $\gamma_B^{-1}=\sigma^{\nu}_{-i/2}
\circ R^{-1}$ and $\gamma_C^{-1}=R\circ\sigma^{\nu}_{-i/2}$. 
\end{proof}

The significance of Proposition~\ref{modularauto} is that the maps $\gamma_B\circ\gamma_C$ and $\gamma_B^{-1}
\circ\gamma_C^{-1}$ provide ``modular automorphisms'' for the weights $\mu$ and $\nu$, respectively. 
In particular, recall Lemma~\ref{lemKMS}.  According to Proposition~\ref{modularauto}, for 
$c\in{\mathcal D}(\gamma_B\circ\gamma_C)={\mathcal D}(\sigma^{\mu}_{-i})$ and for $x\in{\mathfrak M}_{\mu}$, 
we have: $cx,x\sigma^{\mu}_{-i}(c)\in{\mathfrak M}_{\mu}$, and $\mu(cx)=\mu\bigl(x\sigma^{\mu}_{-i}(c)\bigr)
=\mu\bigl(x(\gamma\circ\gamma')(c)\bigr)$.  Similar for the weight $\nu$.

We will conclude this section by showing a few different characterizations of the idempotent $E$.

\begin{prop}\label{sigmasigmaE}
For any $t\in\mathbb{R}$, we have:
$$
(\sigma^{\nu}_t\otimes\sigma^{\mu}_{-t})(E)=E.
$$
\end{prop}

\begin{proof}
This is Proposition~3.7 of \cite{BJKVD_SepId}.  By a straightforward calculation using the properties 
of $\nu$ and $\gamma_B$, we can show that for any $b\in{\mathcal D}(\sigma^{\nu}_{i/2})$, 
\begin{align}
(\nu\otimes\operatorname{id})\bigl((\sigma^{\nu}_t\otimes\sigma^{\mu}_{-t})(E)(b\otimes1)\bigr)
&=\sigma^{\mu}_{-t}\bigl(\gamma_B(\sigma^{\nu}_{-t}(b))\bigr)   \notag \\
&=(R\circ\sigma^{\nu}_{i/2})(b)=(\nu\otimes\operatorname{id})\bigl(E(b\otimes1)\bigr).
\notag
\end{align}
By the uniqueness of $E$ (Proposition~\ref{E_unique}), the result follows.
\end{proof}

Write $\sigma$ to denote the flip map on $M(B\otimes C)$.  Then $\sigma E\in M(C\otimes B)$. 
In the below, we wish to show that $(\gamma_C\otimes\gamma_B)(\sigma E)=E$.  However, as of now, 
we do not know if the expression $(\gamma_C\otimes\gamma_B)(\sigma E)$ even makes sense as 
a bounded element.  To make sense of this, and anticipating other future applications, we consider 
the following lemma:

\begin{lem}\label{thelemma}
Suppose $b\in{\mathcal T}_{\nu}$ and $c\in{\mathcal T}_{\mu}$.  We have:
\begin{align}
&(\gamma_C\otimes\gamma_B)\bigl((\gamma_C^{-1}(b)\otimes\gamma_B^{-1}(c))(\sigma E)\bigr)
=E(b\otimes c),  \notag \\
&(\gamma_C\otimes\gamma_B)\bigl((\sigma E)(\gamma_C^{-1}(b)\otimes\gamma_B^{-1}(c))\bigr)
=(b\otimes c)E,  \notag \\
&(\gamma_C\otimes\gamma_B)\bigl((1\otimes\gamma_B^{-1}(c))(\sigma E)
(\gamma_C^{-1}(b)\otimes1)\bigr)=(b\otimes1)E(1\otimes c),  \notag \\
&(\gamma_C\otimes\gamma_B)\bigl((\gamma_C^{-1}(b)\otimes1)(\sigma E)
(1\otimes\gamma_B^{-1}(c))\bigr)=(1\otimes c)E(b\otimes1).  \notag
\end{align}
\end{lem}

\begin{proof}
This is Lemma~3.8 of \cite{BJKVD_SepId}.  The proof uses the properties of $\gamma_B$ and 
$\gamma_C$, the characterization of the analytic generators  $\sigma^{\mu}_{-i}$ and $\sigma^{\nu}_{-i}$ 
(Proposition~\ref{modularauto}), and the faithfulness of the weights $\mu$ and $\nu$.
\end{proof}

We can now make sense of the expression $(\gamma_C\otimes\gamma_B)(\sigma E)$ 
as a left and right multiplier map, on a dense subset of $B\otimes C$.  By Lemma~\ref{thelemma}, 
we have, for $b\in{\mathcal T}_{\nu}$ (dense in $B$) and 
$c\in{\mathcal T}_{\mu}$ (dense in $C$),
\begin{align}
\bigl[(\gamma_C\otimes\gamma_B)(\sigma E)\bigr](b\otimes c)&:=(\gamma_C\otimes\gamma_B)
\bigl((\gamma_C^{-1}(b)\otimes\gamma_B^{-1}(c))(\sigma E)\bigr)=E(b\otimes c),  \notag \\
(b\otimes c)\bigl[(\gamma_C\otimes\gamma_B)(\sigma E)\bigr]&:=(\gamma_C\otimes\gamma_B)
\bigl((\sigma E)(\gamma_C^{-1}(b)\otimes\gamma_B^{-1}(c))\bigr)=(b\otimes c)E.
\notag
\end{align}
The following proposition is an immediate consequence:

\begin{prop}\label{sigmaE}
\begin{enumerate}
  \item $(\gamma_C\otimes\gamma_B)(\sigma E)=E$ and $(\gamma_B\otimes\gamma_C)(E)
=\sigma E$.
  \item $(R^{-1}\otimes R)(\sigma E)=E$ and $(R\otimes R^{-1})(E)=\sigma E$.
\end{enumerate}
\end{prop}

\begin{proof}
This is Proposition~3.9 of \cite{BJKVD_SepId}.  

(1). We observed above that $(\gamma_C\otimes\gamma_B)(\sigma E)$ coincides with $E$, 
as a left and right multiplier map on a dense subset of $B\otimes C$.  Since $E$ is bounded, 
this implies that $(\gamma_C\otimes\gamma_B)(\sigma E)$ can be canonically extended to 
a left and right multiplier map on all of $B\otimes C$, and that $(\gamma_C\otimes\gamma_B)
(\sigma E)=E\in M(B\otimes C)$. By taking the flip map, we also have: $(\gamma_B\otimes\gamma_C)(E)
=\sigma E$.

(2). Apply $R^{-1}\otimes R$ to both sides of $\sigma E=(\gamma_B\otimes\gamma_C)(E)$. 
Since we know $\gamma_B=R\circ\sigma^{\nu}_{i/2}$ and $\gamma_C=R^{-1}\circ\sigma^{\mu}_{-i/2}$ 
(Proposition~\ref{gammagamma'formuals}), we have:
$$
(R^{-1}\otimes R)(\sigma E)=(R^{-1}\otimes R)\bigl[(\gamma_B\otimes\gamma_C)(E)\bigr]
=(\sigma^{\nu}_{i/2}\otimes\sigma^{\mu}_{-i/2})(E)=E.
$$
By taking the flip map, we also have: $(R\otimes R^{-1})(E)=\sigma E$.
\end{proof}

\section{Comultiplication and the canonical idempotent $E$}

\subsection{Comultiplication}

Our goal is to define a suitable definition of a locally compact quantum groupoid, and of course, 
the first task is to clarify the notion of the comultiplication.  However, unlike in the quantum 
group theory, we cannot assume the non-degeneracy of the comultiplication map.  As observed in 
Introduction, this implies that the usual coassociativity condition of $(\Delta\otimes\operatorname{id})
\Delta=(\operatorname{id}\otimes\Delta)\Delta$ cannot be made sense (yet).  This explains why 
we are only requiring the weak coassociativity condition in the definition below.  The discussions 
in this section are same in spirit as in \cite{VDWangwha0} (in particular, section 1 and Appendix there), 
but adjusted accordingly to fit our $C^*$-algebraic setting.

\begin{defn}\label{comultiplication}
Let $A$ be a $C^*$-algebra.  By a {\em comultiplication\/} on $A$, we mean a ${}^*$-homomorphism 
$\Delta:A\to M(A\otimes A)$ such that
\begin{itemize}
  \item $(\Delta a)(1\otimes b)\in A\otimes A$ and $(a\otimes 1)(\Delta b)\in A\otimes A$, 
  for all $a,b\in A$;
  \item The following ``weak coassociativity condition'' holds:
$$
(a\otimes1\otimes1)\bigl((\Delta\otimes\operatorname{id})((\Delta b)(1\otimes c))\bigr)
=\bigl((\operatorname{id}\otimes\Delta)((a\otimes1)(\Delta b))\bigr)(1\otimes1\otimes c), 
$$
for all $a,b,c\in A$;
  \item The following subspaces are norm-dense in $A$:

$\operatorname{span}\bigl\{(\operatorname{id}\otimes\omega)((\Delta x)(1\otimes y)):\omega\in A^*,\ x,y\in A\bigr\}$,

$\operatorname{span}\bigl\{(\omega\otimes\operatorname{id})((x\otimes1)(\Delta y)):\omega\in A^*,\ x,y\in A\bigr\}$.
\end{itemize}
\end{defn}

\begin{rem}
The first condition originally appeared for multiplier Hopf algebras \cite{VD_multiplierHopf}, and 
is satisfied in the commutative case of $A=C_0(G)$.  It allows both the second and third conditions 
to make sense.  As noted above, we require only the weak coassociativity condition, because we do not 
have non-degeneracy assumption on $\Delta$ for it to be extended to $M(A)$.  The third condition is 
among the defining conditions for a locally compact quantum group \cite{KuVa}.  Often, motivated by 
the terminology in the purely algebraic setting, we refer to this condition as $\Delta$ being ``full''.  
\end{rem}

\begin{lem}\label{LemmaDeltafull}
The following results are also true:
\begin{itemize}
  \item $(1\otimes b)(\Delta a)\in A\otimes A$ and $(\Delta b)(a\otimes 1)\in A\otimes A$, 
  for all $a,b\in A$
  \item The following subspaces are norm-dense in $A$:

$\operatorname{span}\bigl\{(\operatorname{id}\otimes\omega)((1\otimes y)(\Delta x)):\omega\in A^*,\ x,y\in A\bigr\}$,

$\operatorname{span}\bigl\{(\omega\otimes\operatorname{id})((\Delta y)(x\otimes1)):\omega\in A^*,\ x,y\in A\bigr\}$.
\end{itemize}
\end{lem}

\begin{proof}
Just use the ${}^*$-operation on $A$, and note that $\Delta$ is a ${}^*$-homomorphism.
\end{proof}

In the quantum group case, as a consequence of the density condition as in Definition~\ref{comultiplication} 
($\Delta$ being ``full''), it follows that $\Delta(A)(A\otimes A)$ and $(A\otimes A)\Delta(A)$ are dense in 
$A\otimes A$. That is when $\Delta$ is non-generate, which is no longer true in our more general case. 
We will actually require that $\Delta(A)(A\otimes A)$ and $(A\otimes A)\Delta(A)$ are dense in $E(A\otimes A)$ 
and $(A\otimes A)E$, respectively, with a suitable projection $E\in M(A\otimes A)$, the ``canonical idempotent''. 

It will later turn out that this $E$ carries information about the ``source'' and ``target'' algebras, as well as the 
comultiplication map.  Before giving the precise definition of the canonical idempotent, let us gather some useful 
results.

\begin{prop}\label{Deltaproperty}
Let $(A,\Delta)$ be given as in Definition~\ref{comultiplication}.  Suppose there exists a self-adjoint idempotent element $E\in M(A\otimes A)$ 
such that 
$$\overline{\Delta(A)(A\otimes A)}^{\|\ \|}=E(A\otimes A),\quad {\text { and }}\quad 
\overline{(A\otimes A)\Delta(A)}^{\|\ \|}=(A\otimes A)E.
$$
Then the following results hold:
\begin{enumerate}
  \item For any $a\in A$, we have
$$
E(\Delta a)=\Delta a=(\Delta a)E.
$$
  \item Let $x,y\in M(A\otimes A)$ be such that $x(\Delta a)=y(\Delta a)$, for all $a\in A$. 
Then $xE=yE$.  Similarly, if $x,y\in M(A\otimes A)$ is such that $(\Delta a)x=(\Delta a)y$, 
$\forall a\in A$, then $Ex=Ey$.
  \item $E$ is uniquely determined.
  \item There exists a unique ${}^*$-homomorphism $\tilde{\Delta}:M(A)\to M(A\otimes A)$ such that 
  $\tilde{\Delta}|_A=\Delta$ and $\tilde{\Delta}(1_{M(A)})=E$.
\end{enumerate}
\end{prop}

\begin{proof}
(1). For $b,c\in A$, we will have $(\Delta a)(b\otimes c)\in\Delta(A)(A\otimes A)\subseteq E(A\otimes A)$, 
while $E^2=E$.  So we have: $E(\Delta a)(b\otimes c)=(\Delta a)(b\otimes c)$, true for arbitrary $b,c\in A$. 
This shows that $E(\Delta a)=\Delta a$, as elements in $M(A\otimes A)$.  Similarly, we can show that 
$(\Delta a)E=\Delta a$.

(2). Suppose $x(\Delta a)=y(\Delta a)$, $\forall a\in A$.  Since $E(\Delta a)=\Delta a$, we then have 
$xE(\Delta a)=yE(\Delta a)$.  Thus we have:
$$
xE(\Delta a)(a'\otimes a'')=yE(\Delta a)(a'\otimes a''),\quad\forall a,a',a''\in A.
$$
But we know that $\Delta(A)(A\otimes A)$ is dense in $E(A\otimes A)$, while $E^2=E$.  It is thus easy 
to see that $xE=yE$ in $M(A\otimes A)$.  The other statement is proved in a similar way, now using 
the density of $(A\otimes A)\Delta(A)$ in $(A\otimes A)E$.

(3). Suppose $E'$ is another such projection.  From
$$E(A\otimes A)=\overline{\Delta(A)(A\otimes A)}^{\|\ \|}=E'(A\otimes A),$$
since $E'E'=E'$, we will have: $(E'E-E)(A\otimes A)=0$, which means $E'E=E$.  Similarly, we can show $EE'=E$. 
This means that as idempotents, we have: $E\le E'$.  Reversing the roles, we will also have: $E'\le E$, 
so $E'=E$.

(4). Let $m\in M(A)$.  Note first that if such a map $\tilde{\Delta}$ were to exist, we will have, 
for any $x\in A\otimes A$,
$$
\tilde{\Delta}(m)x=\tilde{\Delta}(m\cdot1)x=\tilde{\Delta}(m)\tilde{\Delta}(1)x=\tilde{\Delta}(m)Ex.
$$
By assumption, we know that any $Ex$, $x\in A\otimes A$, can be approximated by elements of the 
form $\sum_{i=1}^N\Delta(a_i)z_i$, where $a_i\in A$, $z_i\in A\otimes A$.  For such elements, 
at least formally, it makes sense to consider:
\begin{equation}\label{(L_m)}
L_{\tilde{\Delta}(m)}:\sum_{i=1}^N\Delta(a_i)z_i\mapsto\tilde{\Delta}(m)
\sum_{i=1}^N\Delta(a_i)z_i=\sum_{i=1}^N\Delta(ma_i)z_i\,\in A\otimes A.
\end{equation}

Similarly, for any $y\in A\otimes A$, we will have $y\tilde{\Delta}(m)=yE\tilde{\Delta}(m)$, and any $yE$ 
is approximated by the elements of the form $\sum_{j=1}^{N'}w_j\Delta(b_j)$, where $b_j\in A$, 
$w_j\in A\otimes A$.  For these elements, we may consider:
\begin{equation}\label{(R_m)}
R_{\tilde{\Delta}(m)}:\sum_{j}w_j\Delta(b_j)\mapsto\sum_j w_j\Delta(b_j)\,\tilde{\Delta}(m)
=\sum_jw_j\Delta(b_jm)\,\in A\otimes A.
\end{equation}

We plan to show that $L_{\tilde{\Delta}(m)}$ and $R_{\tilde{\Delta}(m)}$ are valid linear maps that can 
be naturally extended to maps on $A\otimes A$, and that $(L_{\tilde{\Delta}(m)},R_{\tilde{\Delta}(m)})$ 
forms a double centralizer.  This would show that $\tilde{\Delta}(m)\in M(A\otimes A)$.  Moreover, by 
construction, the uniqueness of the extension map $\tilde{\Delta}:m\mapsto\tilde{\Delta}(m)$ would 
immediately follow.

To see if Equation~\eqref{(L_m)} indeed determines a well-defined linear map, suppose $\sum_{i=1}^N
\Delta(a_i)z_i=0$.  Then for any $u\in A\otimes A$ and $b\in A$, we have:
\begin{align}
u\Delta(b)\,L_{\tilde{\Delta}(m)}\left(\sum_{i=1}^N\Delta(a_i)z_i\right)
&=u\Delta(b)\sum_{i=1}^N\Delta(ma_i)z_i=u\sum_{i=1}^N\Delta(bma_i)z_i  \notag \\
&=u\Delta(bm)\sum_{i=1}^N\Delta(a_i)z_i=0.  \notag
\end{align}
Since $\overline{(A\otimes A)\Delta(A)}^{\|\ \|}=(A\otimes A)E$ and since $E\Delta(a)=\Delta(a)$ 
for any $a\in A$, this implies that $yL_{\tilde{\Delta}(m)}\left(\sum_{i=1}^N\Delta(a_i)z_i\right)
=yE\,L_{\tilde{\Delta}(m)}\left(\sum_{i=1}^N\Delta(a_i)z_i\right)=0$, for any $y\in A\otimes A$. 
So $L_{\tilde{\Delta}(m)}\left(\sum_{i=1}^N\Delta(a_i)z_i\right)=0$.  In this way, we can see that 
$L_{\tilde{\Delta}(m)}:\Delta(A)(A\otimes A)\to A\otimes A$ is a well-defined linear map. 

This map $L_{\tilde{\Delta}(m)}$ is clearly continuous, and extends to a bounded map 
$L_{\tilde{\Delta}(m)}:E(A\otimes A)\to A\otimes A$.  This in turn defines a map 
$L_{\tilde{\Delta}(m)}:A\otimes A\to A\otimes A$, by $L_{\tilde{\Delta}(m)}(x)
:=L_{\tilde{\Delta}(m)}(Ex)$.

In a similar way, Equation~\eqref{(R_m)} defines a bounded map $R_{\tilde{\Delta}(m)}:A\otimes A
\to A\otimes A$ such that $R_{\tilde{\Delta}(m)}(y):=R_{\tilde{\Delta}(m)}(yE)$.

Also from above, observe that: 
$\bigl[\sum_{j=1}^{N'}w_j\Delta(b_j)\bigr]L_{\tilde{\Delta}(m)}\left(\sum_{i=1}^N\Delta(a_i)z_i\right)
=\sum_{j=1}^{N'}\sum_{i=1}^Nw_j\Delta(b_jma_i)z_i
=R_{\tilde{\Delta}(m)}\left(\sum_{j=1}^{N'}w_j\Delta(b_j)\right)\bigl[\sum_{i=1}^N\Delta(a_i)z_i\bigr]$.
In this way, we see that for any $x,y\in A\otimes A$, we have:
\begin{equation}
y\tilde{L}_{\tilde{\Delta}(m)}(x)=yEL_{\tilde{\Delta}(m)}(Ex)=R_{\tilde{\Delta}(m)}(yE)Ex
=\tilde{R}_{\tilde{\Delta}(m)}(y)x.
\end{equation}
This means that the pair $(L_{\tilde{\Delta}(m)},R_{\tilde{\Delta}(m)})$ is a double centralizer. 
By the standard theory, we can canonically associate to it an element of $M(A\otimes A)$,  
which is none other than $\tilde{\Delta}(m)$.

By construction, we see easily that $\tilde{\Delta}(m)\tilde{\Delta}(m')=\tilde{\Delta}(mm')$, 
and that $\tilde{\Delta}(m^*)x=\bigl(x^*\tilde{\Delta}(m)\bigr)^*$ for all $x\in A\otimes A$. 
This verifies that $\tilde{\Delta}$ is a ${}^*$-homomorphism.  Finally, when $m=1$, 
note that $\tilde{\Delta}(1)x=\tilde{\Delta}(1)(Ex)=Ex$, for all $x\in A\otimes A$. 
It follows that $\tilde{\Delta}(1)=E$.
\end{proof}\

For convenience, we will just write $\Delta$ for the unique extended map $\tilde{\Delta}$ obtained above. 
Meanwhile, we can also extend the maps $\Delta\otimes\operatorname{id}$ and $\operatorname{id}
\otimes\Delta$, by using the same method (see below).

\begin{cor}
Let $E$ be as above.  Then the maps $\Delta\otimes\operatorname{id}$ and $\operatorname{id}\otimes\Delta$ 
have unique extensions to ${}^*$-homomorphisms from $M(A\otimes A)$ to $M(A\otimes A\otimes A)$, which 
will be still denoted by $\Delta\otimes\operatorname{id}$ and $\operatorname{id}\otimes\Delta$, such that 
$$(\Delta\otimes\operatorname{id})(1\otimes1)=E\otimes1,\quad {\text{ and }}\quad
(\operatorname{id}\otimes\Delta)(1\otimes1)=1\otimes E,$$
respectively.
\end{cor}

\begin{proof}
Use the same method of proof given in (4) of Proposition~\ref{Deltaproperty} above: We just replace $A$ 
by $A\otimes A$ and $A\otimes A$ by $A\otimes A\otimes A$, with the projection $E\otimes1$ taking the 
role of $E$.  Namely,
$$
(E\otimes1)(\Delta\otimes\operatorname{id})(x)=(\Delta\otimes\operatorname{id})(x)=
(\Delta\otimes\operatorname{id})(x)(E\otimes1),\quad\forall x\in A\otimes A.
$$
In this way, we can (uniquely) define the extended ${}^*$-homomorphism $\tilde{\Delta}\otimes\operatorname{id}$, 
which we may just write as $\Delta\otimes\operatorname{id}$.

Similarly, with $1\otimes E$ taking the role of $E$, we can extend $\operatorname{id}\otimes\Delta$ to 
the level of $M(A\otimes A)$.
\end{proof}

Because of the extension, it is now possible to make sense of the maps $(\Delta\otimes\operatorname{id})
\Delta$ and $(\operatorname{id}\otimes\Delta)\Delta$.  In Theorem~\ref{coassociativity} below, we show 
that the weak coassociativity condition given earlier can now be replaced by the more familiar form of 
the coassociativity.  The method is essentially not much different from the one given in Appendix of 
\cite{VDWangwha0}.  First, we begin with a lemma:

\begin{lem}\label{lemcoassoc}
For clarity, write $\tilde{\Delta}\otimes\operatorname{id}$ and $\operatorname{id}\otimes\tilde{\Delta}$ 
for the extensions of $\Delta\otimes\operatorname{id}$ and $\operatorname{id}\otimes\Delta$ to $M(A\otimes A)$.
Then for any $a,b\in A$, we have:
\begin{align}
&\bigl((\tilde{\Delta}\otimes\operatorname{id})(\Delta a)\bigr)(1\otimes1\otimes b)
=(\Delta\otimes\operatorname{id})\bigl((\Delta a)(1\otimes b)\bigr),  \notag \\
&(a\otimes1\otimes1)\bigl((\operatorname{id}\otimes\tilde{\Delta})(\Delta b)\bigr)
=(\operatorname{id}\otimes\Delta)\bigl((a\otimes1)(\Delta b)\bigr).
\notag
\end{align}
\end{lem}

\begin{proof}
Consider $\sum_i u_i(\Delta\otimes\operatorname{id})(v_i)$, where $u_i\in A\otimes A\otimes A$ and 
$v_i\in A\otimes A$.  Such elements form a dense subspace in $(A\otimes A\otimes A)(E\otimes1)$. 
Observe:
\begin{align}
&\sum_i u_i(\Delta\otimes\operatorname{id})(v_i)\bigl((\tilde{\Delta}\otimes\operatorname{id})(\Delta a)\bigr)
(1\otimes1\otimes b)  \notag \\
&=\sum_i u_i(\Delta\otimes\operatorname{id})\bigl(v_i(\Delta a)\bigr)(1\otimes1\otimes b)
=\sum_i u_i(\Delta\otimes\operatorname{id})\bigl(v_i(\Delta a)(1\otimes b)\bigr)  \notag \\
&=\sum_i u_i(\Delta\otimes\operatorname{id})(v_i)(\Delta\otimes\operatorname{id})
\bigl((\Delta a)(1\otimes b)\bigr).
\notag
\end{align}
Here we used the fact that $(\Delta a)(1\otimes b)\in A\otimes A$.
This means that for any $z\in A\otimes A\otimes A$, we have: 
\begin{align}
z\bigl((\tilde{\Delta}\otimes\operatorname{id})(\Delta a)\bigr)(1\otimes1\otimes b)
&=z(E\otimes1)\bigl((\tilde{\Delta}\otimes\operatorname{id})(\Delta a)\bigr)(1\otimes1\otimes b)
\notag \\
&=z(E\otimes1)(\Delta\otimes\operatorname{id})\bigl((\Delta a)(1\otimes b)\bigr)   \notag \\
&=z(\Delta\otimes\operatorname{id})\bigl((\Delta a)(1\otimes b)\bigr),
\notag
\end{align}
obtaining the first result. The proof for the second result is similar.
\end{proof}

We are now able to prove the coassociativity of $\Delta$:

\begin{theorem}\label{coassociativity}
Let $A$ be a $C^*$-algebra and let $\Delta$ be a comultiplication on it, satisfying the conditions of 
Definition~\ref{comultiplication}.  Suppose also the existence of an idempotent element $E\in M(A\otimes A)$, 
as in Proposition~\ref{Deltaproperty}.  Then the following ``coassociativity condition'' holds:
$$
(\Delta\otimes\operatorname{id})(\Delta a)=(\operatorname{id}\otimes\Delta)(\Delta a),\quad 
{\text { for any $a\in A$.}}
$$
\end{theorem}

\begin{proof}
Let $a,a',a''\in A$.  By the result of Lemma~\ref{lemcoassoc}, we have:
$$
(a'\otimes1\otimes1)\bigl((\Delta\otimes\operatorname{id})(\Delta a)\bigr)(1\otimes1\otimes a'')
=(a'\otimes1\otimes1)(\Delta\otimes\operatorname{id})\bigl((\Delta a)(1\otimes a'')\bigr),
$$
and also
$$
(a'\otimes1\otimes1)\bigl((\operatorname{id}\otimes\Delta)(\Delta a)\bigr)(1\otimes1\otimes a'')
=(\operatorname{id}\otimes\Delta)\bigl((a'\otimes1)(\Delta a)\bigr)(1\otimes1\otimes a'').
$$
But the expressions in the right hand sides are equal, by the weak coassociativity 
(see Definition~\ref{comultiplication}).  Therefore, we have:
$$
(a'\otimes1\otimes1)\bigl((\Delta\otimes\operatorname{id})(\Delta a)\bigr)(1\otimes1\otimes a'')
=(a'\otimes1\otimes1)\bigl((\operatorname{id}\otimes\Delta)(\Delta a)\bigr)(1\otimes1\otimes a''), 
$$
true for arbitrary $a',a''\in A$.  This means that 
$$
(\Delta\otimes\operatorname{id})(\Delta a)=(\operatorname{id}\otimes\Delta)(\Delta a),\quad
\forall a\in A.
$$
\end{proof}

We expect the coassociativity of $\Delta$ to extend to the multiplier algebra level.  See results 
below:

\begin{prop}\label{coassociativityM(A)}
Let $(A,\Delta)$ and $E$ be as above.  Then we have:
\begin{enumerate}
  \item $(\Delta\otimes\operatorname{id})(E)=(\operatorname{id}\otimes\Delta)(E)$;
  \item $(\Delta\otimes\operatorname{id})(\Delta m)=(\operatorname{id}\otimes\Delta)(\Delta m)$, 
for any $m\in M(A)$.
\end{enumerate}
\end{prop}

\begin{proof}
(1). By Proposition~\ref{Deltaproperty}, we have: 
\begin{align}
[(\Delta\otimes\operatorname{id})(E)](A\otimes A\otimes A)
&=[(\Delta\otimes\operatorname{id})(E)](E\otimes1)(A\otimes A\otimes A)  \notag \\
&=[(\Delta\otimes\operatorname{id})(E)]\bigl(\Delta(A)\otimes A\bigr)(A\otimes A\otimes A) \notag \\
&=(\Delta\otimes\operatorname{id})\bigl(E(A\otimes A)\bigr)(A\otimes A\otimes A) \notag \\
&=\bigl((\Delta\otimes\operatorname{id})\Delta(A)\bigr)(A\otimes A\otimes A).
\notag
\end{align}
Since $E$ and $(\Delta\otimes\operatorname{id})(E)$ are projections, all the spaces above 
are norm-closed.  In a similar way, we can also show that
$$
[(\operatorname{id}\otimes\Delta)(E)](A\otimes A\otimes A)=\cdots
=\bigl((\operatorname{id}\otimes\Delta)\Delta(A)\bigr)(A\otimes A\otimes A).
$$
By Theorem~\ref{coassociativity}, the right sides are same.  So 
$[(\Delta\otimes\operatorname{id})(E)](A\otimes A\otimes A)
=[(\operatorname{id}\otimes\Delta)(E)](A\otimes A\otimes A)$. 
Similarly, we have: $(A\otimes A\otimes A)[(\Delta\otimes\operatorname{id})(E)]
=(A\otimes A\otimes A)[(\operatorname{id}\otimes\Delta)(E)]$.  Therefore, 
$(\Delta\otimes\operatorname{id})(E)=(\operatorname{id}\otimes\Delta)(E)$.

(2). Write $G=(\Delta\otimes\operatorname{id})(E)=(\operatorname{id}\otimes\Delta)(E)$, which 
is a projection.  By modifying the proof of Proposition~\ref{Deltaproperty}\,(4) a little, 
with $G$ taking the role of $E$, we can regard the map $M(A)\ni m\mapsto(\Delta\otimes
\operatorname{id})(\Delta m)$ as a (unique) extension of the map $A\ni a\mapsto(\Delta\otimes
\operatorname{id})(\Delta a)$.  Similarly, we can regard the map $M(A)\ni m\mapsto(\operatorname{id}
\otimes\Delta)(\Delta m)$ as a (unique) extension of the map $A\ni a\mapsto(\operatorname{id}
\otimes\Delta)(\Delta a)$, using the same projection $G$.  By Theorem~\ref{coassociativity}, 
the maps at the level of $A$ coincide.  Therefore, the uniqueness means that we have
$(\Delta\otimes\operatorname{id})(\Delta m)=(\operatorname{id}\otimes\Delta)(\Delta m)$, 
$\forall m\in M(A)$.
\end{proof}

\subsection{The canonical idempotent}

In Definition~\ref{canonicalidempotent} below, we formally give the definition of the canonical 
idempotent $E$.  Condition~(1) says that $E$ satisfies the density conditions given in 
Proposition~\ref{Deltaproperty}; Condition~(2) means $E$ is a ``separability idempotent'', 
as discussed in Section~2; Condition~(3) is motivated by the theory of weak (multiplier) 
Hopf algebras \cite{BNSwha1}, \cite{NVfqg}, \cite{VDWangwha0}, \cite{VDWangwha1}. 
See remark following the definition.

\begin{defn}\label{canonicalidempotent}
Let $A$ be a $C^*$-algebra and let $\Delta$ be a comultiplication on it, as in 
Definition~\ref{comultiplication}.  By a {\em canonical idempotent\/} for $(A,\Delta)$, 
we mean a self-adjoint idempotent element $E\in M(A\otimes A)$ such that 
\begin{enumerate}
  \item $\Delta(A)(A\otimes A)$ is norm-dense in $E(A\otimes A)$ and $(A\otimes A)\Delta(A)$ 
is norm-dense in $(A\otimes A)E$;
  \item There exist two non-degenerate $C^*$-subalgebras $B$, $C$ of $M(A)$, a ${}^*$-anti-isomorphism 
  $R=R_{BC}:B\to C$, a KMS weight $\nu$ on $B$, such that $E\in M(B\otimes C)$ and that $(E,B,\nu)$ 
  forms a separability triple, in the sense of Definition~\ref{separabilitytriple};
  \item $E\otimes1$ and $1\otimes E$ commute, and we also have:
$$
(\operatorname{id}\otimes\Delta)(E)=(E\otimes1)(1\otimes E)=(1\otimes E)(E\otimes1)
=(\Delta\otimes\operatorname{id})(E).
$$
\end{enumerate}
\end{defn}

\begin{rem}
We already discussed about the density condition (1), the separability idempotent condition 
(2), as well as their consequences.  See Sections \S3.2 and \S2, respectively.  By (1) and (2), 
if such $E$ exists, it is unique.  Since $B$, $C$ are non-degenerate $C^*$-subalgebras of $M(A)$, 
it follows that $M(B)$, $M(A)$ are also subalgebras of $M(A)$, and $M(B\otimes C)$ is 
a subalgebra of $M(A\otimes A)$.  So $E\in M(B\otimes C)\subseteq M(A\otimes A)$. 
In Proposition~\ref{Efull}, we noted that the source algebra $B$ is a left leg of $E$ and 
the target algebra $C$ is a right leg of $E$.  

Since condition (1) enables us to extend $\Delta$ to the multiplier algebra level and the 
coassociativity holds, we can make sense of the expressions in condition (3).  Condition (3) 
is referred to as the ``weak comultiplicativity of the unit'' in \cite{BNSwha1}.  It is 
always true in the groupoid case, and there are no known cases where this condition is 
violated.  The condition seems natural in that sense, and parts of it can be proved (like 
in Proposition~\ref{coassociativityM(A)}), but apparently, does not seem to follow from 
other axioms.
\end{rem}

Here are some more consequences of the existence of the canonical idempotent $E$. 

\begin{prop}
The $C^*$-algebras $B$ and $C$ commute.
\end{prop}

\begin{proof}
From (3) of Definition~\ref{canonicalidempotent}, we know $(E\otimes1)(1\otimes E)
=(1\otimes E)(E\otimes1)$.  Let $c\in C$ and $\omega\in C^*$.  Then
$$
(\operatorname{id}\otimes\operatorname{id}\otimes\omega)\bigl((E\otimes1)(1\otimes E)
(1\otimes 1\otimes c)\bigr)=(\operatorname{id}\otimes\operatorname{id}\otimes\omega)
\bigl((1\otimes E)(E\otimes1)(1\otimes 1\otimes c)\bigr),
$$
which becomes:
$$
E\bigl(1\otimes(\operatorname{id}\otimes\omega)(E(1\otimes c))\bigr)
=\bigl(1\otimes(\operatorname{id}\otimes\omega)(E(1\otimes c))\bigr)E.
$$
But, due to $E$ being full (see Proposition~\ref{Efull}), we know that the elements 
of the form $(\operatorname{id}\otimes\omega)\bigl(E(1\otimes c)\bigr)$ are dense in $B$. 
So we have: $E(1\otimes b)=(1\otimes b)E$, for all $b\in B$.

Now suppose $b'\in B$ and $\theta\in B^*$ be arbitrary.  Since $E(1\otimes b)=(1\otimes b)E$, 
we have:
$$
(\theta\otimes\operatorname{id})\bigl((b'\otimes1)E(1\otimes b)\bigr)
=(\theta\otimes\operatorname{id})\bigl((b'\otimes b)E\bigr),
$$
which becomes: $(\theta\otimes\operatorname{id})\bigl((b'\otimes1)E\bigr)b=
b(\theta\otimes\operatorname{id})\bigl((b'\otimes1)E\bigr)$.  But we know that the elements 
of the form $(\theta\otimes\operatorname{id})\bigl((b'\otimes1)E\bigr)$ are dense in $C$, 
again by Proposition~\ref{Efull}.  It follows that $cb=bc$, true for $\forall b\in B$, $\forall c\in C$.
\end{proof}

\begin{cor}
$M(B)$ and $M(C)$ also commute.
\end{cor}

\begin{proof}
Straightforward from the previous result.
\end{proof}

In Proposition~\ref{Delta_on_BandC} and its Corollary below, we show how the comultiplication 
behaves on the source and the target algebras.  In fact, it will be shown later 
(Part~II: see \cite{BJKVD_qgroupoid2}) that these properties essentially characterize the 
algebras $M(B)$ and $M(C)$.

\begin{prop}\label{Delta_on_BandC}
We have:
\begin{itemize}
  \item $\Delta b=E(1\otimes b)=(1\otimes b)E$, for all $b\in B$;
  \item $\Delta c=(c\otimes1)E=E(c\otimes1)$, for all $c\in C$.
\end{itemize}
\end{prop}

\begin{proof}
Let $a\in A$ be arbitrary, and let $c\in C$.  Then, from $(\Delta\otimes\operatorname{id})(E)
=(E\otimes1)(1\otimes E)$, we will have:
\begin{align}
(\Delta\otimes\operatorname{id})\bigl((a\otimes c)E\bigr)&=(\Delta a\otimes c)
(\Delta\otimes\operatorname{id})(E)  \notag \\
&=(\Delta a\otimes c)(E\otimes1)(1\otimes E)=(\Delta a\otimes c)(1\otimes E).
\notag
\end{align}
Let $\omega\in C^*$, and apply $\operatorname{id}\otimes\operatorname{id}\otimes\omega$ 
to both side.  Then it becomes:
$$
\Delta\bigl(a(\operatorname{id}\otimes\omega)((1\otimes c)E)\bigr)=
(\Delta a)\bigl(1\otimes(\operatorname{id}\otimes\omega)((1\otimes c)E)\bigr).
$$
But, by Proposition~\ref{Efull}, elements of the form $(\operatorname{id}\otimes\omega)
\bigl((1\otimes c)E\bigr)$, $c\in C$, $\omega\in C^*$ are dense in $B$.  It follows that
$\Delta(ab)=(\Delta a)\bigl(1\otimes b\bigr)$, for all $b\in B$.  Since this is true 
for any $a\in A$, and since $\Delta(ab)=(\Delta a)(\Delta b)$, we may apply 
Proposition~\ref{Deltaproperty}\,(2) to see that $\Delta b=E(1\otimes b)$, true 
for all $b\in B$. 
Moreover, since $E\in M(B\otimes C)$ and since we know that $M(B)$ and $M(C)$ commute, 
it follows that $E(1\otimes b)=(1\otimes b)E$, thereby proving the first statement.

The proof for $\Delta c=(c\otimes1)E=E(c\otimes1)$, $\forall c\in C$, is done similarly.
\end{proof}

\begin{cor}
We have:
\begin{itemize}
  \item $\Delta y=E(1\otimes y)=(1\otimes y)E$, for all $y\in M(B)$;
  \item $\Delta x=(x\otimes1)E=E(x\otimes1)$, for all $x\in M(C)$.
\end{itemize}
\end{cor}

\begin{proof}
Let $y\in M(B)$.  Then for any $b\in B$, we have $yb\in B$.  Then by 
Proposition~\ref{Delta_on_BandC}, 
$$
\Delta(yb)=E(1\otimes yb)=E(1\otimes y)(1\otimes b).
$$
On the other hand,
$$
\Delta(yb)=(\Delta y)(\Delta b)=(\Delta y)E(1\otimes b)=(\Delta y)(1\otimes b).
$$
It follows that $(\Delta y)(1\otimes b)=E(1\otimes y)(1\otimes b)$, for any $b\in B$. 
But note that $B$ is a non-degenerate $C^*$-subalgebra of $M(A)$.  Thus it follows 
easily that $(\Delta y)(1\otimes m)=E(1\otimes y)(1\otimes m)$, for any $m\in M(A)$. 
In other words, $\Delta y=E(1\otimes y)$.  As before, using the fact that $M(B)$ 
and $M(C)$ commute, we also have: $\Delta y=(1\otimes y)E$.

The proof for $\Delta x=(x\otimes1)E=E(x\otimes1)$, $\forall x\in M(C)$, is similar.
\end{proof}

\section{Left and Right Haar weights}

Assume that we are given $(A,\Delta,E,B,\nu)$, where $A$ is a $C^*$-algebra; $\Delta$ 
is a comultiplication on $A$; $B$ is a non-degenerate $C^*$-subalgebra of $M(A)$; 
$\nu$ is a KMS weight on $B$; and $E$ is the canonical idempotent element 
as described in the previous section.  

Following the spirit of the authors' previous paper at the level of weak multiplier 
Hopf algebras \cite{BJKVD_LSthm} (see also the discussion given in Introduction), 
we will further require the existence of a left invariant weight $\varphi$ and 
a right invariant weight $\psi$ on $(A,\Delta)$.  The definitions for the invariant 
weights (``Haar weights'') will be given in this section, together with their 
properties.  

\subsection{Definition of the invariant weights $\varphi$ and $\psi$}

In the commutative case when $A=C_0(G)$, the weights $\varphi$ and $\psi$ are the 
functionals that come from combining the left/right Haar systems with the weight 
$\nu$ at the level of the unit space.  Recall the discussion given in Introduction, 
including the definition of the left invariance given in 
Equation~\eqref{(groupoid_leftinvariance)} and the definition of the functional 
$\varphi$ given in Equation~\eqref{(groupoid_leftHaar)}.  The following definition 
may not look obvious, but turns out to be a suitable generalization to our non-commutative 
setting.

\begin{defn}\label{leftrightinvariance}
Let the notation be as above.
\begin{enumerate}
  \item A KMS weight $\varphi$ on $A$ is said to be {\em left invariant\/}, 
  if for any $a\in{\mathfrak M}_{\varphi}$, we have: 
$\Delta a\in\overline{\mathfrak M}_{\operatorname{id}\otimes\varphi}$ and 
$(\operatorname{id}\otimes\varphi)(\Delta a)\in M(C)$.
  \item A KMS weight $\psi$ on $A$ is said to be {\em right invariant\/}, 
  if for any $a\in{\mathfrak M}_{\varphi}$, we have: 
$\Delta a\in\overline{\mathfrak M}_{\psi\otimes\operatorname{id}}$ and 
$(\psi\otimes\operatorname{id})(\Delta a)\in M(B)$.
\end{enumerate}
\end{defn}

\begin{prop}\label{invarianceprop}
Suppose $\varphi$ and $\psi$ are left and right invariant weights, respectively, for 
$(A,\Delta)$.  Then as a consequence of Definition~\ref{leftrightinvariance}, 
we have:
\begin{itemize}
  \item If $a\in{\mathfrak N}_{\varphi}$, we have: 
$\Delta a\in\overline{\mathfrak N}_{\operatorname{id}\otimes\varphi}$.
  \item If $a\in{\mathfrak N}_{\psi}$, we have: 
$\Delta a\in\overline{\mathfrak N}_{\psi\otimes\operatorname{id}}$.
\end{itemize}
\end{prop}

\begin{proof}
We will just prove the first statement.  Note that by definition, we have 
$a\in{\mathfrak N}_{\varphi}$ if and only if $a^*a\in{\mathfrak M}_{\varphi}^+$. 
Then by Definition~\ref{leftrightinvariance}, we have $\Delta(a^*a)
\in\overline{\mathfrak M}_{\operatorname{id}\otimes\varphi}$.  But $\Delta(a^*a)
=\Delta(a)^*\Delta(a)$, because $\Delta$ is a ${}^*$-homomorphism.  It follows 
that $\Delta(a)\in\overline{\mathfrak N}_{\operatorname{id}\otimes\varphi}$.
\end{proof}

\begin{rem}
In the commutative case, the expression $(\operatorname{id}\otimes\varphi)(\Delta a)$ 
corresponds to $\int_{G^{s(x)}}a(xz)\,d\lambda^{s(x)}(z)$, because the product $xz$ 
is valid only when $t(z)=s(x)$.  The left invariance condition for an ordinary locally compact 
groupoid, given in Equation~\eqref{(groupoid_leftinvariance)}, says that this is equal to 
$\int_{G^{t(x)}}a\,d\lambda^{t(x)}$.  But the map $x\mapsto\int_{G^{t(x)}}a\,d\lambda^{t(x)}$ 
is really a function in $t(x)$, so contained in $M(C)$.  Considering this, the left invariance 
condition given above in Definition~\ref{leftrightinvariance} makes sense.  In fact, except 
for the issues concerning the unboundedness, this is essentially the way the left/right 
invariance condition is formulated in the theory of weak multiplier Hopf algebras and 
the algebraic quantum groupoids \cite{VDWangwha3}.

In a different approach, in the measured quantum groupoid framework \cite{LesSMF}, 
\cite{EnSMF}, the left/right invariance conditions are formulated in terms of 
operator-valued weights $T_L$ and $T_R$.  In our case, they would correspond to 
$a\mapsto(\operatorname{id}\otimes\varphi)(\Delta a)$ and $a\mapsto(\psi\otimes
\operatorname{id})(\Delta a)$.  In the next subsection (Proposition~\ref{nuphipsi}), 
we further discuss the relationships between these operator-valued weights and 
our weights $\varphi$ and $\psi$.
\end{rem}

In what follows, assume $\psi$ is a right invariant weight.  We gather some technical results 
that follow as a consequence.  

\begin{prop}\label{idomegadelta}
If $a\in{\mathfrak M}_{\psi}$ (so $\Delta a\in\overline{\mathfrak M}_{\psi\otimes\operatorname{id}}$) 
and $\omega\in A^*$, then we have $(\operatorname{id}\otimes\omega)(\Delta a)
\in\overline{\mathfrak M}_{\psi}$ and $\psi\bigl((\operatorname{id}\otimes\omega)(\Delta a)\bigr)
=\omega\bigl((\psi\otimes\operatorname{id})(\Delta a)\bigr)$.
\end{prop}

\begin{proof}
Since $\Delta a\in\overline{\mathfrak M}_{\psi\otimes\operatorname{id}}$, this is a consequence 
of Proposition~\ref{slice_lem1}.
\end{proof}

\begin{lem}\label{omegabar}
\begin{enumerate}
  \item Let $\omega$ be a continuous linear functional on $A$.  Then there exists a unique 
positive linear functional $|\omega|$ on $A$ such that $\bigl\||\omega|\bigr\|=\|\omega\|$ 
and $\overline{\omega(a)}\omega(a)=\bigl|\omega(a)\bigr|^2\le\|\omega\|\,|\omega|(a^*a)$, 
for $a\in A$.
  \item For $z\in M(A\otimes A)$ and $\omega\in A^*$, we have:
$$
(\operatorname{id}\otimes\omega)(z)^*(\operatorname{id}\otimes\omega)(z)\le\|\omega\|
\bigl(\operatorname{id}\otimes|\omega|\bigr)(z^*z).
$$
\end{enumerate}
\end{lem}

\begin{proof}
(1). This is a standard result.  See, for instance, Proposition~3.6.7 of \cite{Pedersenbook}.

(2). This is a consequence of (1).  For proof, see Lemma~3.10 of \cite{KuVaweightC*}.
\end{proof}

\begin{prop}\label{prop_omegabar}
If $a\in{\mathfrak N}_{\psi}$ and $\omega\in A^*$, then $(\operatorname{id}\otimes\omega)
(\Delta a)\in\overline{\mathfrak N}_{\psi}$ and we have: 
$$
\psi\bigl((\operatorname{id}\otimes\omega)(\Delta a)^*(\operatorname{id}\otimes\omega)
(\Delta a)\bigr)\le\|\omega\|\,|\omega|\bigl((\psi\otimes\operatorname{id})(\Delta(a^*a))\bigr).
$$
\end{prop}

\begin{proof}
By the right invariance of $\psi$ (see Proposition~\ref{invarianceprop}), we know that 
$\Delta a\in\overline{\mathfrak N}_{\psi\otimes\operatorname{id}}$ and that 
$(\operatorname{id}\otimes\omega)(\Delta a)\in\overline{\mathfrak N}_{\psi}$.  So 
$(\operatorname{id}\otimes\omega)(\Delta a)^*(\operatorname{id}\otimes\omega)
(\Delta a)\in\overline{\mathfrak M}_{\psi}$.  But by Lemma~\ref{omegabar}, we have:
$$
(\operatorname{id}\otimes\omega)(\Delta a)^*(\operatorname{id}\otimes\omega)(\Delta a)
\le\|\omega\|\bigl(\operatorname{id}\otimes|\omega|\bigr)\bigl(\Delta(a^*a)\bigr).
$$
Apply $\psi$ to both sides.  By Proposition~\ref{idomegadelta}, we have 
$\psi\bigl((\operatorname{id}\otimes|\omega|)(\Delta(a^*a))\bigr)
=|\omega|\bigl((\psi\otimes\operatorname{id})(\Delta(a^*a))\bigr)$, from which 
the result of the proposition follows.
\end{proof}

The following result is a consequence of the generalized Cauchy--Schwarz inequality 
given in Proposition~\ref{cauchyschwarz}.

\begin{prop}\label{delta_cs}
Let $a\in{\mathfrak N}_{\psi}$, $b\in\overline{\mathfrak N}_{\psi}$, and $x\in M(A\otimes A)$. 
Then we have $\Delta(a^*)x(b\otimes1)\in\overline{\mathfrak M}_{\psi\otimes\operatorname{id}}$, 
and
\begin{align}
&(\psi\otimes\operatorname{id})\bigl(\Delta(a^*)x(b\otimes1)\bigr)^*
\,(\psi\otimes\operatorname{id})\bigl(\Delta(a^*)x(b\otimes1)\bigr)   \notag \\
&\le\,\bigl\|(\psi\otimes\operatorname{id})(\Delta(a^*a))\bigr\|
\,(\psi\otimes\operatorname{id})\bigl((b^*\otimes1)x^*x(b\otimes1)\bigr).
\notag
\end{align}
\end{prop}

\begin{proof}
We know $\Delta a\in\overline{\mathfrak N}_{\psi\otimes\operatorname{id}}$, from the 
right invariance of $\psi$.  Meanwhile, since $\overline{\mathfrak N}_{\psi}$ is a 
left ideal, we have $x(b\otimes1)\in\overline{\mathfrak N}_{\psi\otimes\operatorname{id}}$. 
So $\Delta(a^*)x(b\otimes1)=(\Delta a)^*\,[x(b\otimes1)]\in\overline{\mathfrak M}_{\psi\otimes
\operatorname{id}}$.  Now apply the result of the generalized Cauchy--Schwarz inequality 
given in Proposition~\ref{cauchyschwarz}.
\end{proof}

Of course, analogous results as in Propositions~\ref{idomegadelta}, \ref{prop_omegabar}, \ref{delta_cs} 
will hold for the left invariant weight $\varphi$.  All these results will be useful later.  Meanwhile, 
the next result involves both weights $\psi$ (right invariant) and $\varphi$ (left invariant).

\begin{prop}
Let $a,b\in{\mathfrak N}_{\psi}$ and $c\in{\mathfrak N}_{\varphi}$.  Then
\begin{itemize}
  \item $\Delta(a^*c)(b\otimes1)=\Delta(a^*)(\Delta c)(b\otimes1)
\in\overline{\mathfrak M}_{\psi\otimes\operatorname{id}}$;
  \item $(\psi\otimes\operatorname{id})\bigl(\Delta(a^*c)(b\otimes1)\bigr)
\in\overline{\mathfrak N}_{\varphi}$.
\end{itemize}
\end{prop}

\begin{proof}
First statement is an immediate consequence of Proposition~\ref{delta_cs}.  The proposition 
also says
\begin{align}
&(\psi\otimes\operatorname{id})\bigl(\Delta(a^*c)(b\otimes1)\bigr)^*
\,(\psi\otimes\operatorname{id})\bigl(\Delta(a^*c)(b\otimes1)\bigr)  \notag \\
&\le\bigl\|(\psi\otimes\operatorname{id})(\Delta(a^*a))\bigr\|
\,(\psi\otimes\operatorname{id})\bigl((b^*\otimes1)\Delta(c^*c)(b\otimes1)\bigr).
\notag
\end{align}
Apply here $\omega\in{\mathcal G}_{\varphi}$.  Then it becomes:
\begin{align}
&\omega\bigl((\psi\otimes\operatorname{id})(\Delta(a^*c)(b\otimes1))^*
\,(\psi\otimes\operatorname{id})(\Delta(a^*c)(b\otimes1))\bigr)  \notag \\
&\le\bigl\|(\psi\otimes\operatorname{id})(\Delta(a^*a))\bigr\|
\,\psi\bigl(b^*(\operatorname{id}\otimes\omega)(\Delta(c^*c))b\bigr).
\notag
\end{align}
This is true for any $\omega\in{\mathcal G}_{\varphi}$.  Meanwhile, by the left invariance 
of $\varphi$, we know $\Delta(c^*c)\in\overline{\mathfrak M}_{\operatorname{id}\otimes\varphi}$. 
So by the lower semi-continuity of $\varphi$, the inequality will still hold for $\varphi$ 
in place of $\omega$.  Then we have:
\begin{align}
&\bigl|\varphi\bigl((\psi\otimes\operatorname{id})(\Delta(a^*c)(b\otimes1))^*
\,(\psi\otimes\operatorname{id})(\Delta(a^*c)(b\otimes1))\bigr)\bigr|  \notag \\
&\le\bigl\|(\psi\otimes\operatorname{id})(\Delta(a^*a))\bigr\|
\,\bigl\|(\operatorname{id}\otimes\varphi)(\Delta(c^*c))\bigr\|
\,\bigl|\psi(b^*b)\bigr|\,<\infty.
\notag
\end{align}
This shows that $(\psi\otimes\operatorname{id})\bigl(\Delta(a^*c)(b\otimes1)\bigr)
\in\overline{\mathfrak N}_{\varphi}$.  To be more precise, we have:
\begin{align}
&\bigl\|\Lambda_{\varphi}\bigl((\psi\otimes\operatorname{id})(\Delta(a^*c)
(b\otimes1))\bigr)\bigr\|_{{\mathcal H}_{\varphi}}   \notag \\
&\le\bigl\|(\psi\otimes\operatorname{id})(\Delta(a^*a))\bigr\|^{\frac12}
\,\bigl\|(\operatorname{id}\otimes\varphi)(\Delta(c^*c))\bigr\|^{\frac12}
\,\bigl\|\Lambda_{\psi}(b)\bigr\|_{{\mathcal H}_{\psi}}.
\notag
\end{align}
\end{proof}

Observe that the results gathered so far in this subsection follow directly from the 
left/right invariance of $\varphi$, $\psi$, and we do not need to require the existence 
of the canonical idempotent $E$.

\subsection{Definition of a locally compact quantum groupoid}

We are trying to build our case that the data $(A,\Delta,E,B,\nu,\varphi,\psi)$ determines 
a valid locally compact quantum groupoid.  The complete justification will have to 
wait until Part~II, but we are now ready to state the full definition.

\begin{defn}\label{definitionlcqgroupoid}
The data $(A,\Delta,E,B,\nu,\varphi,\psi)$  defines a {\em locally compact quantum 
groupoid\/}, if
\begin{itemize}
  \item $A$ is a $C^*$-algebra;
  \item $\Delta:A\to M(A\otimes A)$ is a comultiplication on $A$, in the sense of 
Definition~\ref{comultiplication};
  \item $B$ is a non-degenerate $C^*$-subalgebra of $M(A)$;
  \item $\nu$ is a KMS weight on $B$;
  \item $E$ is the canonical idempotent of $(A,\Delta)$ as in 
Definition~\ref{canonicalidempotent}.  Namely,
  \begin{enumerate}
    \item $\Delta(A)(A\otimes A)$ is norm-dense in $E(A\otimes A)$ and $(A\otimes A)\Delta(A)$ 
is norm-dense in $(A\otimes A)E$;
    \item there exists a $C^*$-algebra $C\cong B^{\operatorname{op}}$ with 
a ${}^*$-anti-isomorphism $R=R_{BC}:B\to C$ so that $E\in M(B\otimes C)$, and the triple 
$(E,B,\nu)$ forms a separability triple in the sense of Definition~\ref{separabilitytriple};
    \item $E\otimes1$ and $1\otimes E$ commute, and we have:
$$
(\operatorname{id}\otimes\Delta)(E)=(E\otimes1)(1\otimes E)=(1\otimes E)(E\otimes1)
=(\Delta\otimes\operatorname{id})(E).
$$
  \end{enumerate}
  \item $\varphi$ is a KMS weight, and is left invariant (Definition~\ref{leftrightinvariance});
  \item $\psi$ is a KMS weight, and is right invariant (Definition~\ref{leftrightinvariance});
  \item There exists a one-parameter group of automorphisms 
$(\theta_t)_{t\in\mathbb{R}}$ of $B$ such that $\nu\circ\theta_t=\nu$ and that 
$\sigma^{\varphi}_t|_B=\theta_t$, $\forall t\in\mathbb{R}$.
\end{itemize}
\end{defn}

\begin{rem}
If $B=C=\mathbb{C}$, then necessarily $E=1\otimes1$ and $\Delta$ becomes 
a non-degenerate ${}^*$-homomorphism.  The left invariance condition will become 
$(\operatorname{id}\otimes\varphi)(\Delta a)\in\mathbb{C}$, $a\in{\mathfrak M}_{\varphi}$. 
In this case, with an additional normalization requirement that 
$(\operatorname{id}\otimes\varphi)(\Delta a)=\varphi(a)1$, and similarly 
for $\psi$, we will obtain the definition of a locally compact quantum group 
given by Kustermans and Vaes \cite{KuVa}, together with the uniqueness 
(up to scalar multiplication) of the weights $\varphi$ and $\psi$.  In general, 
however, when the base $C^*$-algebra $B$ is non-trivial, we would no longer have 
the uniqueness result for the Haar weights.
\end{rem}

In the upcoming paper (Part~II), we show that the data 
$(A,\Delta,E,B,\nu,\varphi,\psi)$ indeed allow us to construct the antipode map $S$. 
This would justify that we have a valid framework for locally compact quantum groupoids.

On the other hand, the existence requirement of the (bounded) canonical idempotent 
element $E$ means that the definition is not fully general.  The existence of $E$ is really 
a condition imposed on the pair $(B,\nu)$.  See \cite{BJKVD_SepId}.  If we are to overcome 
this restriction and develop a theory without requiring the existence of $E$, we will have to 
consider the comultiplication as a map from $A$ into a kind of a ``balanced tensor product'' 
$A\ast A$ (involving the base algebra $B$), whose definition is not yet clear in the $C^*$-algebra 
framework.  In the purely algebraic framework of {\em multiplier Hopf algebroids\/} \cite{TimVD}, 
the ``Takeuchi product'' is used, while in the von Neumann algebraic setting of {\em measured 
quantum groupoids\/} \cite{LesSMF}, \cite{EnSMF}, the ``fiber product'' (over a relative tensor 
product of Hilbert spaces) is used, which does not have a $C^*$-algebraic analog.  See 
discussion given in Introduction.

Meanwhile, the last condition, saying that the the restriction to $B$ of the automorphism 
group $(\sigma^{\varphi}_t)$ is invariant under $\nu$, is needed as one wishes to prove later 
the ``quasi-invariance'' of $\nu$.  Namely, we aim to prove as a consequence of the axioms that 
the automorphism groups $(\sigma^{\varphi}_t)$ and $(\sigma^{\psi}_t)$ commute, which is 
typically required already in the ordinary locally compact groupoid theory.  This condition 
cannot be really removed, since it cannot be proved using the other axioms.  A similar 
requirement appeared in Lesieur's definition of a measured quantum groupoid \cite{LesSMF}. 
More discussion on this condition will be given later, in Part~II \cite{BJKVD_qgroupoid2}.

Let $(A,\Delta,E,B,\nu,\varphi,\psi)$ be a locally compact quantum groupoid, as in 
Definition~\ref{definitionlcqgroupoid}.  The next proposition shows the relationships between 
the weights $\psi$, $\varphi$ with the expressions $(\psi\otimes\operatorname{id})(\Delta x)$ 
and $(\operatorname{id}\otimes\varphi)(\Delta x)$, which are in fact operator-valued weights. 
These results are analogous to Equation~\eqref{(groupoid_leftHaar)} in the case 
of ordinary groupoids, further strengthening the case for our choice of the right/left 
invariance conditions.

\begin{prop}\label{nuphipsi}
We have:
\begin{itemize}
  \item $\nu\bigl((\psi\otimes\operatorname{id})(\Delta x)\bigr)=\psi(x)$, 
for $x\in{\mathfrak M}_{\psi}$.
  \item $\mu\bigl((\operatorname{id}\otimes\varphi)(\Delta x)\bigr)=\varphi(x)$, 
for $x\in{\mathfrak M}_{\varphi}$.
\end{itemize}
\end{prop}

\begin{proof}
Let $a\in{\mathfrak M}_{\psi}$.  Then by the right invariance of $\psi$, we know 
$\Delta a\in\overline{\mathfrak M}_{\psi\otimes\operatorname{id}}$ and 
$(\psi\otimes\operatorname{id})(\Delta a)\in M(B)$.  Apply here $\Delta$. 
On the one hand, we have:
$$
\Delta\bigl((\psi\otimes\operatorname{id})(\Delta a)\bigr)
=(\psi\otimes\operatorname{id}\otimes\operatorname{id})\bigl((\operatorname{id}
\otimes\Delta)(\Delta a)\bigr)
=(\psi\otimes\operatorname{id}\otimes\operatorname{id})\bigl((\Delta
\otimes\operatorname{id})(\Delta a)\bigr),
$$
where we used the coassociativity of $\Delta$.  Meanwhile, by Corollary to 
Proposition~\ref{Delta_on_BandC}, we have: 
$$
\Delta\bigl((\psi\otimes\operatorname{id})(\Delta a)\bigr)
=E\bigl(1\otimes(\psi\otimes\operatorname{id})(\Delta a))
=(\psi\otimes\operatorname{id}\otimes\operatorname{id})
\bigl((1\otimes E)\Delta_{13}(a)\bigr).
$$
It thus follows that 
\begin{equation}\label{(nuphipsi)}
(\psi\otimes\operatorname{id}\otimes\operatorname{id})\bigl((\Delta
\otimes\operatorname{id})(\Delta a)\bigr)
=(\psi\otimes\operatorname{id}\otimes\operatorname{id})
\bigl((1\otimes E)\Delta_{13}(a)\bigr).
\end{equation}

Let $y=\tilde{y}c$, where $\tilde{y}\in A$ and $c\in D(\gamma_C)$, which is 
dense in $C$.  Since $C$ is a non-degenerate $C^*$-subalgebra of $M(A)$, 
such elements form a dense subset in $A$.  Multiply $1\otimes y=1\otimes
\tilde{y}c$ to both sides of Equation~\eqref{(nuphipsi)}, from left.  Then 
the equation becomes: 
$$
(\psi\otimes\operatorname{id}\otimes\operatorname{id})
\bigl((\Delta\otimes\operatorname{id})((1\otimes\tilde{y}c)(\Delta a))\bigr)
=(\psi\otimes\operatorname{id}\otimes\operatorname{id})
\bigl((1\otimes1\otimes\tilde{y}c)(1\otimes E)\Delta_{13}(a)\bigr).
$$
Let $\omega\in A^*$, and apply $\operatorname{id}\otimes\omega$ to the 
equation above.  Then it becomes: 
\begin{align}
&(\psi\otimes\operatorname{id})\bigl(\Delta((\operatorname{id}\otimes\omega)
[(1\otimes\tilde{y}c)(\Delta a)])\bigr)   \notag \\
&=(\psi\otimes\operatorname{id})\bigl((\operatorname{id}\otimes\operatorname{id}
\otimes\omega)[(1\otimes1\otimes\tilde{y})(1\otimes1\otimes c)(1\otimes E)
\Delta_{13}(a)]\bigr).
\notag
\end{align}
Since $c\in D(\gamma_C)$, we have $(1\otimes c)E\in{\mathfrak M}_{\nu\otimes
\operatorname{id}}$, and $(\nu\otimes\operatorname{id})\bigl((1\otimes c)E\bigr)
=c$.  In particular, we can see that the right side of the above equation is 
actually contained in ${\mathfrak M}_{\nu}$.  So let us apply $\nu$ to both sides 
of the equation.  Then: 
\begin{align}
&\nu\bigl((\psi\otimes\operatorname{id})(\Delta((\operatorname{id}\otimes\omega)
[(1\otimes\tilde{y}c)(\Delta a)]))\bigr)   \notag \\
&=\psi\bigl((\operatorname{id}\otimes\omega)[(1\otimes\tilde{y})(1\otimes
(\nu\otimes\operatorname{id})((1\otimes c)E))(\Delta a)]\bigr)   \notag \\
&=\psi\bigl((\operatorname{id}\otimes\omega)[(1\otimes\tilde{y}c)(\Delta a)]\bigr).
\notag
\end{align}
For convenience, write: $x=(\operatorname{id}\otimes\omega)[(1\otimes\tilde{y}c)
(\Delta a)]$.  Then this result can be written as  $\nu\bigl((\psi\otimes
\operatorname{id})(\Delta x)\bigr)=\psi(x)$.

Note that by Lemma~\ref{LemmaDeltafull}, the elements of the form 
$(\operatorname{id}\otimes\omega)[(1\otimes\tilde{y}c)(\Delta a)]$ are dense in $A$. 
What all this means is that the equation $\nu\bigl((\psi\otimes\operatorname{id})
(\Delta x)\bigr)=\psi(x)$ is true for any $x\in A$, as long as the expression 
makes sense.  In particular, if $x\in{\mathfrak M}_{\psi}$, making the right 
side to be valid, we would have $(\psi\otimes\operatorname{id})(\Delta x)
\in\overline{\mathfrak M}_{\nu}$ and the equation would hold.

Similarly, we can show that if $x\in{\mathfrak M}_{\varphi}$ then 
$(\operatorname{id}\otimes\varphi)(\Delta x)\in\overline{\mathfrak M}_{\mu}$, and 
that $\mu\bigl((\operatorname{id}\otimes\varphi)(\Delta x)\bigr)=\varphi(x)$.
\end{proof}

In Part~II, to justify that the definition of a locally compact quantum groupoid 
$(A,\Delta,E,B,\nu,\varphi,\psi)$ as given in Definition~\ref{definitionlcqgroupoid} 
is a valid one, we will carry out the constructions for left/right regular 
representations (in terms of certain ``multiplicative partial isometries'') and 
the antipode map (expressed using a polar decomposition).  The overall look is 
similar to the case of locally compact quantum groups \cite{KuVa}, \cite{KuVavN}, 
\cite{MNW}, \cite{VDvN}, which is to be expected.  But of course, the current 
case is more general, since $E\ne1\otimes1$.

\section{Examples}

At this stage, we have only stated the definition of our quantum groupoid, with 
the main theory and construction postponed to upcoming papers.  As this is the case, 
we do not plan to give any lengthy discussion on examples.  Instead, we will 
mention here only a few relatively straightforward ones.  Once we further develop 
the theory (in Part~II and Part~III), we plan to explore some more examples that 
fit our framework.

\subsection{Finite quantum groupoids}

The notion of {\em finite quantum groupoids\/} (or {\em weak Hopf $C^*$-algebras\/}) 
goes back to B{\" o}hm, Nill, Szlach{\' a}nyi.  See \cite{BNSwha1}, \cite{NVfqg}, \cite{Valfqg}. 
It is determined by the data $(A,\Delta,\varepsilon,S)$, where $A$ is a finite-dimensional 
$C^*$-algebra; $\Delta:A\to A\otimes A$ is the ``comultiplication'', which is a (not 
necessarily unital) ${}^*$-homomorphism satisfying the coassociativity condition; 
$\varepsilon:A\to\mathbb{C}$ is the ``counit''; and $S:A\to A$ is the ``antipode'', 
together with some number of axioms regarding the maps $\Delta$, $\varepsilon$, $S$.

Then it is known that there exists a unique faithful positive linear functional 
$\varphi:A\to\mathbb{C}$ (the ``normalized Haar measure'') such that $\varphi\circ S
=\varphi$ and $(\operatorname{id}\otimes\varphi)\bigl(\Delta(1)\bigr)=1$.  There 
also exist two counital maps $\varepsilon_s$ and $\varepsilon_t$, given by 
$$
\varepsilon_s(x):=(\operatorname{id}\otimes\varepsilon)\bigl((1\otimes x)\Delta(1)\bigr),
\quad {\text { and }}\quad
\varepsilon_t(x):=(\varepsilon\otimes\operatorname{id})\bigl(\Delta(1)(x\otimes1)\bigr).
$$
It is shown that $S\circ\varepsilon_t=\varepsilon_s\circ S$, and the images of these 
maps become subalgebras of $A$.  Namely, $B=\varepsilon_s(A)$ and $C=\varepsilon_t(A)$.

Write $E:=\Delta(1)$, which is a projection ($E\ne1\otimes1$ in general).  And let 
$\nu:=\varphi|_B$.  Then it is not difficult to show that $(A,\Delta,E,B,\nu,\varphi,\varphi)$ 
determines a locally compact quantum groupoid, satisfying Definition~\ref{definitionlcqgroupoid}.

Even if this is finite-dimensional, there are some interesting examples of weak Hopf 
$C^*$-algebras or finite quantum groupoids, including examples coming from depth-2 
subfactors.  See \cite{NVfqg} and \cite{Valfqg}.

\subsection{Locally compact quantum groups}

The resemblance of our definition to that of {\em locally compact quantum groups\/} in 
the sense of \cite{KuVa} has been noted already.  Indeed, if $(A,\Delta,\varphi,\psi)$ 
determines an arbitrary locally compact quantum group ($A$ does not have to be 
finite-dimensional or compact), then just take $B=\mathbb{C}$ and $E=1\otimes1$ 
to regard it as a locally compact quantum groupoid.  See Remark following 
Definition~\ref{definitionlcqgroupoid}.

This observation agrees well with the case of ordinary groups, which are special cases 
of groupoids (when the set of units contain only one element).  

\subsection{Linking quantum groupoids (De Commer)}

In \cite{DeComm_galois}, \cite{DeComm_LinkingQG}, De Commer considers the notion of 
a {\em linking quantum groupoid\/}, as he establishes some Morita-type correspondence 
relation between quantum groups.  Loosely speaking, he considers quantum groupoids 
with a finite-dimensional base algebra, with the quantum groups on the ``corners''. 
While the overall $C^*$-algebra is not finite-dimensional, the base algebra is, and 
it is possible to show that these examples fit well with our framework. 

More recently, Enock  gave a more general construction, in the setting of measured 
quantum groupoids \cite{En_Morita}.

\subsection{Partial compact quantum groups (De Commer, Timmermann)}

In \cite{DCTT_partialCQG}, De Commer and Timmermann introduced the notion of {\em partial 
compact quantum groups\/}.  In short, a typical partial compact quantum group looks like 
$A=\bigoplus_{k,l,m,n\in I}{}_m\!\!{^k\!}A_n^l$, a direct sum of the vector spaces ${}_m\!\!{^k\!}A_n^l$, 
with indices contained in a discrete set $I$. It would carry a partial bialgebra structure, with the 
partial multiplication $M:{}_m\!\!{^k\!}A^r_s\otimes{}_s\!\!{^r\!}A_n^l\to{}_m\!\!{^k\!}A_n^l$, and the 
partial comultiplications $\Delta_{r,s}:{}_m\!\!{^k\!}A_n^l\to{}_r\!\!{^k\!}A^l_s\otimes{}_m\!\!{^r\!}A_n^s$. 
The index set $I$ plays the role of the ``object set'', and one can also give suitable definitions for 
the ${}^*$-structure, the counit, the antipode, and the invariant integral, making this an algebraic 
quantum groupoid of a weak multiplier Hopf algebra type.

The $C^*$-algebraic partial compact quantum group was studied in \cite{DC_C*partialCQG}.  With 
$E=\Delta(1)$, this would provide an example of a quantum groupoid in our framework developed 
in this paper.  It turns out that there exists a unique left/right invariant weight $\varphi$ (see Theorem~2.11 
of \cite {DCTT_partialCQG}).  It is observed that $\varphi$ is obtained as a direct sum of the (bounded) 
states $\varphi_{km}$ on the subalgebras ${}_m\!\!{^k\!}A_m^k$, so there are less technical issues 
involved.  Dynamical quantum $SU(2)$ group is shown to fit into this framework.

\bigskip\bigskip



\providecommand{\bysame}{\leavevmode\hbox to3em{\hrulefill}\thinspace}
\providecommand{\MR}{\relax\ifhmode\unskip\space\fi MR }
\providecommand{\MRhref}[2]{%
  \href{http://www.ams.org/mathscinet-getitem?mr=#1}{#2}
}
\providecommand{\href}[2]{#2}

\end{document}